\documentclass[12pt,a4paper]{amsart}
\usepackage{amsfonts,amssymb,amsmath,amsthm}
\usepackage{mathdots}
\usepackage{latexsym}
\usepackage[colorlinks,citecolor=magenta,urlcolor=blue,linkcolor=blue]{hyperref}

\newtheorem{theorem}{Theorem}[section]
\newtheorem{lemma}[theorem]{Lemma}
\newtheorem{proposition}[theorem]{Proposition}
\newtheorem{corollary}[theorem]{Corollary}

\newtheorem{remark}[theorem]{Remark}
\newtheorem{notation}[theorem]{Notation}

\newcommand{\fld}{{\mathbb{F}}}

\begin{document}

\title[Invariant Fields of classical Sylow groups]{The Invariant Fields of the Sylow groups of Classical Groups
in the natural characteristic}
\author{Jorge N M Ferreira}
\address{
Centro de Compet\^{e}ncias de Ci\^{e}ncias Exactas e da Engenharia\\
Universidade da Madeira\\
Campus Universit\'{a}rio da Penteada
9000-390 Funchal\\
Portugal}

\thanks{Jorge Ferreira was supported by Funda\c{c}\~{a}o para a Ci\^{e}ncia e Tecnologia for the financial support provided, through Programa Operacional Potencial Humano (POPH) do Fundo Social Europeu (FSE), by the PhD grant SFRH/BD/30132/2006.}

\author{Peter Fleischmann}
\address{
School of Mathematics, Statistics and Actuarial Science\\
Cornwallis Building\\
University of Kent\\
CT2 7NF Canterbury\\
United Kingdom
}

\begin{abstract}
Let $X$ be any finite classical group defined over a
finite field of characteristic $p>0$.  In this paper we
determine the fields of rational invariants for the Sylow
$p$-subgroups of $X$, acting on the natural module.
In particular we prove that these fields are generated
by orbit products of variables and certain invariant
polynomials which are images under Steenrod operations,
applied to the respective invariant linear forms defining
$X$.
\end{abstract}
\maketitle


\section{Introduction}

Let $\fld$ be a field, $V$ a finite dimensional $\fld$-vector space and
$G$ a finite subgroup of ${\rm GL}(V)$. Then $G$ acts naturally on the dual
space $V^*$ of $V$ and therefore on the symmetric algebra
$S:=\mathbb{F}[V]:={\rm Sym}(V^*)$, by graded algebra automorphisms. One of the
main problems of invariant theory is the investigation of the structure
of the \emph{ring of invariants}
$$R:=\mathbb{F}[V]^G:=\{f\in \mathbb{F}[V]\ |\ g\cdot f=f\ \forall\ g\in G\}.$$
Since $G$ is finite it is easy to see that $S$ is a finitely generated $R$-module,
which implies, by a classical result of Emmy Noether, that $R$ is a finitely generated
$\fld$-algebra. Let $\mathbb{L}:={\rm Quot}(S)$ be the quotient field of $S$
and $\mathbb{K}:={\rm Quot}(R)$ the quotient field of $R$. The finiteness of $G$ implies
that $\mathbb{K}=\mathbb{L}^G$ and therefore, by Artin's
main theorem in Galois Theory, that the field extension $\mathbb{L}\ge \mathbb{K}$
is Galois with group $G$. Moreover, it is well known that $R$ is a normal domain,
i.e. $R$ is integrally closed in $\mathbb{K}$.
\\
There are several constructive procedures that, if applied to ring elements
$f\in S$, transform them into invariants in $R$: two examples are
the \emph{transfer}- or \emph{trace map} $f\mapsto {\rm tr}(f)=\sum_{g\in G} g\cdot f$ and
the \emph{norm} $f\mapsto {\rm Norm}(f):=\prod_{g\in G}g\cdot f$. If $|G|$ is a unit
in $\fld$, then ${\rm tr}(S)=R$, otherwise, ${\rm tr}(S)$ is a proper ideal in $R$.
In general, as a result of such operations, one obtains the subalgebra $A\le R$
generated by those invariants, but the major open question remains, when to stop,
i.e. when a full finite set of \emph{generating invariants} of $R$ as a $\fld$-algebra
has been achieved. In certain cases one can use information on the invariant field
$\mathbb{K}$:
For example, the following are equivalent (see \cite{homloc})
\begin{itemize}
\item[(i)] ${\rm Quot}(A) = \mathbb{K}$ and $R$ is integral over $A$;
\item[(ii)] $R$ is the integral closure of $A$ in $\mathbb{L}$;
\item[(iii)] there exists $0\ne a\in A$ with $R=S\cap \frac{1}{a}A$.
\end{itemize}
The intersection in (iii) can in principle be calculated by
Groebner basis methods (\cite{homloc}) and there are also generic
algorithms available to calculate integral closures appearing in (ii)
(see \cite{Jong98}). It is however still a difficult task to determine
$R$ in general, from information on $A$. Nevertheless, in the pursuit of
constructing $R$ it is an important first step to find an explicit description
of the invariant field $\mathbb{K}$. In this paper this is done for all the
$p$-Sylow groups of finite classical groups and $p$ the characteristic of the
field of definition.
\\[2mm]
Before stating the main result in compact form we need a few remarks on the
groups considered, some known results and some ideas that motivated
this work.
\\
Let $X={\rm GL}_n(q)$ with $q=p^s$ acting on $V=\mathbb{F}_q^n$ and
$U$ the Sylow $p$-subgroup of $X$ formed by the upper uni-triangular matrices. Dickson in 1911 proved that the invariant
ring $\mathbb{F}_q[V]^X$ is the polynomial ring  $\mathbb{F}_q[c_0,\ldots, c_{n-1}]$ on generators
$c_i$ of degree $q^n-q^i$ (see \cite{Dickson11}).
Let $U(n,q)$ be the group of lower triangular matrices with ones along the diagonal and $x_1,\ldots,x_n$
a basis for the dual vector space $V^*$. Then $x_1$ is invariant and the orbit of each $x_i$, with $i>1$,
consists of all elements $x_i+w$ where $w$ belongs to the subspace $V_{i-1}$ spanned by $x_1,\ldots,x_{i-1}$.
The orbit product of each $x_i$ is
$N(x_i)=\prod_{w\in V_{i-1}}(x_i+w)=F_{i-1,q}(x_i),$
where $F_{i-1,q}(X)$ is the polynomial (\ref{eq:Fn(X)}) in \emph{section \ref{sec:invfields}}.
It can be easily proven that the polynomials $N(x_i)$ are homogeneous of degree $q^{i-1}$ and the product of their degrees is equal to the order of $U(n,q)$. Applying Theorem \ref{the:polynomialring} we conclude that
$\mathbb{F}_q[V]^{U(n,q)}=\mathbb{F}_q[N(x_1),N(x_2),\ldots,N(x_n)],$
which is a polynomial ring. \\
There is a particularly useful structure, present in invariant theory over the finite field
$\mathbb{F}_q$: Let ${\frak F}:=\mathbb{F}_q[V]$. Then the \emph{$q$ - Steenrod algebra ${\mathcal A}:={\mathcal A}_q$} is the graded $\mathbb{F}_q$-subalgebra
${\mathcal A} = \mathbb{F}\langle{\mathcal P}^i\ |\ i\in \mathbb{N}_0\rangle \le {\rm End}_{\mathbb{F}_q}({\frak F}),$
generated by the homogeneous \emph{Steenrod operators} $\mathcal P^i$ of degree $i(q-1)$, which
themselves are uniquely determined as elements of ${\rm End}_{\mathbb{F}_q}({\frak F})$, by
the following rules:
\begin{enumerate}
\item ${\mathcal P}^0={\rm id}_{{\frak F}}$;
\item the Cartan identity ${\mathcal P}^i(fg)=\sum_{0\le r,s\atop r+s=i} {\mathcal P}^r(f){\mathcal P}^s(g)$;
\item ${\mathcal P}^1(x_j)=x_j^q$ and ${\mathcal P}^k(x_j)=0,\ \forall k>1,j\ge 1$.
\end{enumerate}
The elements ${\mathcal P}^i$ are also uniquely determined by the requirement that
$${\mathcal P}(\zeta):\ {\frak F} \to {\frak F}[[\zeta]],\ f\mapsto \sum_{i\ge 0} {\mathcal P}^i(f)\zeta^i$$
is the unique homomorphism of $\mathbb{F}$ - algebras which maps $v$ to $v+v^q\zeta$ for each
$v\in \langle x_1,x_2,\ldots
,x_n \rangle_{\mathbb{F}}$. From this it is easy to see that
the ${\mathcal A}$ acts on $\mathbb{F}_q[V]$, commuting with the natural action
of ${\rm GL}(V)$. Therefore if $G\le {\rm GL}(V)$, then ${\mathcal A}$ also acts
on $\mathbb{F}_q[V]^G.$
\\
Now let $X$ be any of the following finite classical groups:
\begin{itemize}
\item the general unitary groups $GU(2m,q^2)$ and $GU(2m+1,q^2)$ of dimension
$2m$ and $2m+1$, defined over the field $\mathbb{F}_{q^2}$,
\item the symplectic group $Sp(2m,q)$ of dimension $2m$ over $\mathbb{F}_q$,
\item the general orthogonal groups $O^+(2m,q)$, $O^-(2m+2,q)$ and $O(2m+1,q)$  defined
over $\mathbb{F}_q$.
\end{itemize}
For more details we refer to section four, but typically $X$ is defined as a subgroup
of ${\rm GL}(V)$, fixing a certain form $h\in V^*$ or, in the case of unitary groups,
a homogeneous element $h\in\mathbb{F}_{q^2}[V]$. In other words
$X={\rm Stab}_{{\rm GL}(V)}(h)$, hence for any subgroup $G\le X$,
automatically $h$ is a $G$-invariant and so are the ``Steenrod images" $\mathcal P^i(h)$.
The explicit description of the ring of invariants of the groups $Sp(2m,q)$ (see \cite{CK} and
\cite{DBen93}) and
$GU(n,q^2)$ (see \cite{ChuJow06}) supports the conjecture that invariant rings of classical groups are
always generated by ``Dickson invariants" together with certain Steenrod images $\mathcal P^i(h)$
of the relevant form. Replacing Dickson invariants by ``orbit products of variables"
a similar conjecture can be made about the invariant rings of Sylow $p$-groups of $X$.
We will give some evidence to this by proving the corresponding result for the invariant fields.
\\
It is well known that invariant fields of finite $p$-groups in characteristic $p$ are purely
transcendental (see \cite{Miyata71}). The main result of this paper will describe
transcendence bases consisting of certain explicit orbit products $N(x_i)$, called
``norms" and of invariants $h_i$, which are images of $h$ under certain
Steenrod operators. Now let $G\le X$ be a Sylow $p$-group for $p={\rm char}(\mathbb{F})$.
If $X=O^+(2m,2^e)$ or $X=O^-(2m+2,2^e)$ we also define distinguished maximal subgroups
$G_1\unlhd G$ (see Lemmas \ref{le:invfieldO+evenG1} and \ref{le:invfieldO-evenG1}).
Now set $\frak{G}=G_1$ if $X=O^+(2m,2^e)$ or $X=O^-(2m+2,2^e)$ and $\frak{G}=G$ otherwise.
Then our main result can be stated in short form as follows:

\begin{theorem}\label{main_compact}
The  invariant field $\mathbb{F}(V)^\frak{G}={\rm Quot}(\mathbb{F}[V]^\frak{G})$ is
purely transcendental, generated by $\frak{G}$-orbit products of variables and Steenrod images of
the form $h$ defining $X$.
\end{theorem}

\begin{corollary}\label{main_compact_cor}
Let $R$ be the subalgebra of $\mathbb{F}[V]^\frak{G}$,
generated by $\frak{G}$-orbit products of variables and Steenrod images of
the form $h$ defining $X$. Then
$\mathbb{F}[V]^\frak{G}$ is equal to the integral closure of $R$ in its fraction field.
\end{corollary}
\begin{proof}
By Theorem \ref{main_compact}, $R$ and $\mathbb{F}[V]^\frak{G}$
have the same Quotient field, say $\mathbb{K}$. It follows from \cite{Campbell_Wehlau_book} Theorem 4.0.3. pg. 60,
that $R$ contains a homogeneous system of parameters of $\mathbb{F}[V]^\frak{G}$,
which is therefore integral over $R$.
Let $f\in \mathbb{K}$ be integral over $R$, then $f$ is integral over
$\mathbb{F}[V]^\frak{G}$ and therefore contained in the normal ring $\mathbb{F}[V]^\frak{G}$.
\end{proof}

\begin{remark}
\begin{enumerate}
\item Explicit generators for all $\mathbb{F}(V)^\frak{G}$ are described in Theorems
\ref{the:invariantfieldGUeven}, \ref{the:invariantfieldGUodd},
\ref{the:invfieldsSp},\ref{the:invfieldO+odd},\ref{the:invfieldO-odd}.
\item The generators for $\mathbb{F}(V)^\frak{G}$, when $X=O^+(2m,2^e)$ or $X=O^-(2m+2,2^e)$, are described
in Lemmas \ref{le:invfieldO+evenG1} and \ref{le:invfieldO-evenG1}.
\item Let $x_1,\cdots,x_n$ be a basis of $V^*$,
then one can choose $R=\mathbb{F}[N(x_1),\cdots,N(x_n),h_1,h_2,\cdots,h_k],$
where $N(x_i)$ is the $\frak{G}$-orbit product of $x_i$,
$k=\displaystyle\frac{n}{2}-1$ if $n$ is even, $k=\displaystyle\frac{n-1}{2}$ if $n$ is odd,
and the $h_i$'s are Steenrod images of the form $h$, described in the theorems mentioned above.
\item The generators for $\mathbb{F}(V)^G$, when $X=O^+(2m,2^e)$ or $X=O^-(2m+2,2^e)$, are described
in Theorems \ref{the:invfieldO+even} and \ref{the:invfieldO-even}.
\end{enumerate}
\end{remark}

The precise statements and their proofs need explicit descriptions of the Sylow $p$-groups
and will therefore be formulated after those details have been established.
The further organization of the paper is as follows:
\\[1mm]
In \emph{section one} we will introduce notation and collect some information
due to the special nature of classical groups over finite fields and of
$p$-groups in characteristic $p$.
In \emph{section two} we will give two brief examples in small rank, to illustrate the general strategy of our proof. In \emph{section three} we will develop an explicit description of the Sylow $p$-groups
in terms of (almost always) lower-uni-triangular matrices. Although the Sylow $p$-groups of classical groups are known in principle, their structure is usually described in terms of ``root subgroups", defined in the context of the
theory of finite groups of Lie type (\cite{Carter:b}). Since our methods rely
on explicit calculations, a description in terms of matrices is necessary,
but not easily available in the literature. To avoid unnecessary
repetitions later on our emphasis was to achieve such a description in a form
as unifying as possible. Therefore the results stated in \emph{section $3$} can be
useful for other purposes, that require explicit matrix calculations in those
groups. In \emph{section $4$} we state and prove the precise versions of
Theorem \ref{main_compact} for each Sylow $p$-subgroup.
In \emph{section $5$ and $6$} we present the technical proofs for the results of \emph{section $3$ and $4$} respectively. \\
In special cases of classical Sylow $p$-groups, we have been able to
use the results presented here to determine the rings of invariants $R^G$
for arbitrary $q$, using SAGBI-basis techniques. We will present those
results elsewhere. The general problem of determining
the rings of invariants for all Sylow $p$-groups of classical groups is still unsolved as yet.

\section{Invariant fields of $p$-groups and two small examples}
From now on, throughout the whole paper, let $\fld$ be a field of characteristic $p>0$,
$V$ a finite dimensional $\fld$-vector space and $G\le {\rm GL}(V)$ a finite $p$-group.
We have mentioned that the invariant field $\fld(V)^G$ is always purely transcendental. Moreover,
it turns out that one can construct a transcendence basis consisting of polynomials
in $\fld[V]^G$ algorithmically. This is due to Campbell \& Chuai \cite{CampChuai07} and Kang \cite{Kang06}.
We now present the algorithm as it is described in \cite{CampChuai07}.

Since any $p$-subgroup of $GL(V)$ is triangularizable, there exist a basis $e_1,\ldots,e_n$ for $V$ such that each
element of $G$ is represented by a lower triangular matrix with ones along the diagonal.
Therefore if $x_1,\ldots,x_n$ is the dual basis with respect to $e_1,\ldots,e_n$, then  $(\sigma-1)x_m$ is in the subspace spanned by $x_1,\ldots,x_{m-1}$ for all $\sigma \in G$. From this we can easily see that $x_1$ is invariant.

We define $R[j]:=\mathbb{F}[x_1,\dots,x_j]$ for $0\leq j\leq n$ subject to the convention that $R[0]:=0$. Then  $G$ acts on each ring $R[j]$. For each $j$ we choose an invariant $\phi_j \in R[j]^G$ with the smallest positive degree in $x_j$ among the elements of $R[j]^G$.

\begin{theorem}\label{the:invfieldgen}
Let $G$ be a $p$-group. Then the polynomials $\phi_1,\dots,\phi_n$ defined above generate the invariant field for $G$, i.e.,
$$\mathbb{F}(V)^G=\mathbb{F}(\phi_1,\dots,\phi_n).$$
Moreover, there exists $f\in \mathbb{F}[\phi_1,\dots,\phi_n]$ such that
$$\mathbb{F}[V]^G[f^{-1}]=\mathbb{F}[\phi_1,\dots,\phi_n][f^{-1}].$$
\end{theorem}

\begin{proof}
See Theorem 2.4 in \cite{CampChuai07}.
\end{proof}

We now present two small examples to exemplify the main ideas of the paper and how to use Theorem \ref{the:invfieldgen}.
We shall consider a Sylow $p-$subgroup for $GU(8,q^2)$ and $O^+(8,2^e)$. In the next section we show how to construct the Sylow $p-$subgroups of the finite classical groups.

Let $J_n$ be the matrix (\ref{eq:Jm}) and let $G$ be a Sylow $p-$subgroup for $GU(8,q^2)$. Then we can represent its elements as
\begin{equation}\label{eq:sylGU8}
\left(
\begin{tabular}[h]{c|c|c}
$A$ &$0$ & $0$ \\[-0.4 cm]
& \\
\hline
& \\[-0.4 cm]
$B$ &$F$ & $0$
\\[-0.4 cm]
& \\
\hline
& \\[-0.4 cm]
$J_{3}(\bar{A}^{-1})^{T}\bar{S}$ &$D$ &$J_{3}(\bar{A}^{-1})^{T}J_{3}$
\end{tabular}
\right)
\end{equation}
where
		\begin{itemize}
			\item $A\in U(3,q^2)$, $B$ is any $2\times 3$ matrix and $D=-J_{3}(\bar{A}^{-1})^T\bar{B}^TJ_2F$;
			\item $S$ is a $3\times 3$ matrix such that $S+\bar{S}^T=-B^TJ_2\bar{B}$.
			\item $F=\left(
\begin{tabular}[h]{cc}
$1$ & $0$ \\
$c$ &$1$
\end{tabular}
\right)$
where $c$ is an element in $\mathbb{F}_{q^2}$ satisfying $c+\bar{c}=0$.
\end{itemize}
We will always fix the graded reverse lexicographic order with $x_1<x_2<\cdots <x_8$.
By looking to the elements of $G$ we can see that it acts on $\mathbb{F}_{q^2}[x_1,\ldots,x_5]$ as subgroup $U$ of $U(5,q^2)$. We consider the orbit products $N(x_j)$ for $j\in\{1,\ldots,5\}$.
These are homogeneous polynomials and their degree product is equal to the order of $U$. Hence
$$\mathbb{F}_{q^2}[x_1,\dots,x_{5}]^G=\mathbb{F}_{q^2}[x_1,N(x_2),\ldots,N(x_{5})]$$
which is a polynomial ring. Since in the grevlex order their leading monomials are algebraically independent we can take $\phi_j=N(x_j)$ for $j\in\{1,\ldots,5\}$.
Let $H$ be the abelian subgroup of $G$
$$
\left(
\begin{tabular}[h]{c|c|c}
$I_{3}$ &$0$ & $0$ \\[-0.4 cm]
& \\
\hline
& \\[-0.4 cm]
$0$ &$I_2$ & $0$
\\[-0.4 cm]
& \\
\hline
& \\[-0.4 cm]
$J_{3}\bar{S}$ &$0$ &$I_{3}$
\end{tabular}
\right).
$$
We use subgroups of $H$ to determine a lower bound for the degree in $x_6,x_7$ and $x_8$ of $\phi_6,\phi_7$ and $\phi_8$, respectively.
Let $C=J_{3}\bar{S}$. Then
$$C=\left(
\begin{tabular}{ccc}
$s_{1,3}$& $s_{2,3}$ & $s_{3,3}$ \\
$s_{1,2}$ & $s_{2,2}$ & $-\overline{s}_{2,3}$ \\
$s_{1,1}$ & $-\overline{s}_{1,2}$ & $-\overline{s}_{1,3}$
\end{tabular}
\right).$$
We define $C^{(1)}=C$ and for $k=2,3$, $C^{(k)}$ will be the matrix obtained from $C$ by fixing all the entries of the first $k-1$ rows equal to zero. For each $k$, we denote by $L_k$ the subgroup of $H$ obtained by replacing the matrix $C$ by $C^{(k)}$. Note that $L_1=H$.
The groups $L_k$ act on $R[5+k]$ by fixing $x_1,x_2,x_3,x_4,x_5,\ldots x_{5+k-1}$ and
$x_{5+k}\mapsto x_{5+k}+\sum_{j=1}^{4-k}c_{k,j}x_j.$
Therefore, $L_k$ is acting like a subgroup of $U(5+k,q^2)$ with order $q^{7-2k}$. It is not hard to check that the orbit product of $x_{5+k}$ under $L_k$ has degree $q^{7-2k}$. Hence
$$R[5+k]^{L_k}=\mathbb{F}_{q^2}[x_1,x_2,x_3,x_4,x_5,x_{5+1},\ldots,x_{5+k-1},N(x_{5+k})]$$
and since $R[5+k]^{H}\subset R[5+k]^{L_k}$
we conclude that the minimal degree in $x_{5+k}$ of a polynomial in $R[5+k]^{H}$ is greater or equal to $q^{7-2k}$.
Now we consider the polynomials
\begin{itemize}
    \item $h_1=\Lambda_{1,0}=x_8^{q}x_1+x_8x_1^{q}+x_7^{q}x_2+x_7x_2^{q}+x_6^{q}x_3+x_6x_3^{q}+
x_5^{q}x_4+x_5x_4^{q}$;

    \item $h_2=\Lambda_{2,0}=x_8^{q^3}x_1+x_8x_1^{q^3}+x_7^{q^3}x_2+x_7x_2^{q^3}+x_6^{q^3}x_3+x_6x_3^{q^3}+
x_5^{q^3}x_4+x_5x_4^{q^3}$;

    \item $h_3=\Lambda_{3,0}=x_8^{q^5}x_1+x_8x_1^{q^5}+x_7^{q^5}x_2+x_7x_2^{q^5}+x_6^{q^5}x_3+x_6x_3^{q^5}+
x_5^{q^5}x_4+x_5x_4^{q^5}$.
\end{itemize}
The polynomial $h_1$ comes from the hermitian form we used to define the unitary group, thus it is invariant. The other two being Steenrod images of $h_1$ are invariant also.
By considering the action of the $\mathbb{F}_{q^2}$-algebra endomorphisms $\psi_l$ (see Section \ref{sec:invfields}, Proposition \ref{prop:psil}-2, Lemmas \ref{cor:psilOmegaGammaDeltaalgebramemb}, \ref{le:hkinvariantsGUeven} and Proposistion \ref{prop:degreesphijOmegaGammaDelta}) on $h_1,h_2$ and $h_3$ we get for $k\in\{1,2,3\}$
$$\psi_{3-k}(h_1)\in \mathbb{F}_{q^2}[x_1,x_2,\ldots,x_{5+k}]^G$$
and its degree in $x_{5+k}$ is equal to $q^{7-2k}$ by Proposistion \ref{prop:degreesphijOmegaGammaDelta}. Therefore we can take $\phi_6=\psi_2(h_1)$,$\phi_2=\psi_1(h_1)$ and $\phi_8=\psi_0(h_1)=h_1$. Applying Lemma \ref{cor:psilOmegaGammaDeltaalgebramemb}, we see that
$\psi_l(h_1)\in \mathbb{F}_{q^2}[x_1,N(x_2),\ldots,N(x_l),h_1,h_2,\ldots,h_{l+1}]$
and we get
\begin{theorem}
Let $G$ be the Sylow $p$-subgroup of $GU(8,q^2)$. Then
\[
\mathbb{F}_{q^2}(V)^G=\mathbb{F}_{q^2}(x_1,N(x_2),\ldots,N(x_{5}),h_1,h_2,h_3).
\]
\end{theorem}
Now we consider the orthogonal group $O^+(8,q)$ with $q=2^e$. Let $G_1$ be the subgroup of $U(8,q)$ that preserve the quadratic form. Then its elements can be represent as matrices of type $(\ref{eq:sylGU8})$
where
		\begin{itemize}
			\item $A\in U(3,q)$, $B$ is any $2\times 3$ matrix and $D=-J_{3}(A^{-1})^TB^TJ_2$;
			\item $S$ is a $3\times 3$ matrix such that $S+S^T=B^TJ_2B$ and $s_{ii}=b_{1i}b_{2i}$;
			\item $F$ is the identity matrix.
		\end{itemize}
The group $G_1$ is not a Sylow $p$-subgroup of $O^+(8,q)$. To obtain one, we pick the element
$$L:=\left(
\begin{tabular}[h]{c|c|c}
$I_{3}$ & $0$& $0$ \\[-0.4 cm]
& \\
\hline
& \\[-0.4 cm]
$0$ &$J_2$ & $0$
\\[-0.4 cm]
& \\
\hline
& \\[-0.4 cm]
$0$ & $0$&$I_{3}$
\end{tabular}
\right)$$
in the orthogonal group, which has order $2$ and normalises $G_1$. Then $G=<G_1,L>$ is a Sylow $p$-subgroup of $O^+(8,q)$.
Now, the following polynomials are invariant
\begin{itemize}
    \item $h_1=\Omega_{0,1}=x_8x_1+x_7x_2+x_6x_3+x_5x_4$;

    \item $h_2=\Omega_{1,1}=x_8^qx_1+x_8x_1^q+x_7^qx_2+x_7x_2^q+x_6^qx_3+x_6x_3^q+x_5^qx_4+x_5x_4^q$;

    \item $h_3=\Omega_{2,1}=
x_8^{q^2}x_1+x_8x_1^{q^2}+x_7^{q^2}x_2+x_7x_2^{q^2}+x_6^{q^2}x_3+x_6x_3^{q^2}+
x_5^{q^2}x_4+x_5x_4^{q^2}$.
\end{itemize}
Applying similar arguments as before, we can show that
$$\mathbb{F}_{q}(V)^{G_1}=\mathbb{F}_{q}(x_1,N(x_2),\ldots,N(x_{5}),h_1,h_2,h_3).$$
Using Galois theory we obtain that
$\mathbb{F}_{q}(V)^{G}=(\mathbb{F}_{q}(V)^{G_1})^{<L>}$, i.e.,
the fraction field of $R:=\mathbb{F}_{q}[x_1,N(x_2),\ldots,N(x_{5}),h_1,h_2,h_3]^{<L>}.$
It is not hard to check that $<L>$ will fix the elements
$x_1,N(x_2),N(x_{3})$, $h_1$,$h_2$,$h_{3}$ and swap $N(x_{4})$ with $N(x_{5})$.
Hence,
$$R=\mathbb{F}_{q}[x_1,N(x_2),N(x_{3}),N(x_{4})+N(x_{5}),N(x_{4})N(x_{5}),h_1,h_2,h_3]$$
\begin{theorem}
Let $G$ be the Sylow $p$-subgroup of $O^+(8,q)$. Then
\[
\mathbb{F}_{q}(V)^G=\mathbb{F}_{q}(x_1,N(x_2),N(x_{3}),N(x_{4})+N(x_{5}),N(x_{4})N(x_{5}),h_1,h_2,h_3).
\]
\end{theorem}

\section{Sylow $p$-subgroups}\label{sec:sylowdescription}
Let $\mathbb{F}$ be either the finite field  $\mathbb{F}_q$ or $\mathbb{F}_{q^2}$. Define  the matrix $\bar{A}:=[\bar{a}_{ij}]$ where $\bar{a}_{ij}=a_{ij}^q$.

\begin{notation}
Let $U(n,\mathbb{F})$ the group of $n\times n$ lower triangular matrices with entries in $\mathbb{F}$ and with ones along the diagonal. Also we shall write $M(n\times m,\mathbb{F})$ (or just $M(n,\mathbb{F})$, when $m=n$)  for the set of all $n\times m$ matrices whose entries belong to $\mathbb{F}$. When we want to make clear which field we are working with, we write $U(n,r)$ and $M(n\times m,r)$ (or $M(n,r)$) instead,  $r$ being the number of elements in $\mathbb{F}$.
\end{notation}

Let $\epsilon$ denote one of the two symbols ``$+$" or ``$-$" and
let $X_1\in GL(n,\mathbb{F})$ such that $X_1^2=I$ and let $X_2\in M(l,\mathbb{F})$ satisfying  $X_2=\epsilon\bar{X}_2^T$.
Consider the matrix
\begin{equation}\label{eq:matrixX}
X:=\left(
\begin{tabular}[h]{c|c|c}
$0$ & $0$& $X_1$ \\[-0.4 cm]
& \\
\hline
& \\[-0.4 cm]
$0$ &$X_2$ & $0$
\\[-0.4 cm]
& \\
\hline
& \\[-0.4 cm]
$\epsilon\bar{X}_1^T$ & $0$&$0$
\end{tabular}
\right)
\end{equation}
in $M(2n+l,\mathbb{F})$.
We define the  subgroups
$$\mathfrak G_{X_1,X_2}^{\epsilon}:=\{N\in U(2n+l,\mathbb{F})\ |\
N^TX\bar{N}=X\}.$$
 We write $N\in U(2n+l,\mathbb{F})$ as
$$\left(
\begin{tabular}[h]{c|c|c}
$A$ &$0$ & $0$ \\[-0.4 cm]
& \\
\hline
& \\[-0.4 cm]
$B$ &$F$ & $0$
\\[-0.4 cm]
& \\
\hline
& \\[-0.4 cm]
$C$ &$D$ &$E$
\end{tabular}
\right)
$$
where $A,E \in U(n,\mathbb{F})$, $F\in U(l,\mathbb{F})$, $C\in M(n,\mathbb{F})$, $B\in M(l\times n,\mathbb{F})$ and  $D\in M(n\times l,\mathbb{F})$. Then $N^TX\bar{N}=X$ if
\begin{equation}\label{syst:solution}
\left\{
\begin{array}{l}
D=-\bar{X}_1(\bar{A}^{-1})^T\bar{B}^T\bar{X}_2F\\
F^TX_2\bar{F}=X_2\\
E=\bar{X}_1(\bar{A}^{-1})^{T}\bar{X}_1\\
C=\bar{X}_1(\bar{A}^{-1})^T\bar{S}
\end{array}
\right.
\end{equation}
where $S+(\epsilon\bar{S}^T)=-B^TX_2\bar{B}$.
We shall denote the entries of $S$ and $B$ by $s_{ij}$ and $b_{ij}$, respectively.

\begin{lemma}\label{lem:solutions}
Let $N$ be an element of $U(2n+l,\mathbb{F})$. Then
$N\in \mathfrak G_{X_1,X_2}^{\epsilon}$ if and only if the system (\ref{syst:solution}) holds, with $S+(\epsilon\bar{S}^T)=-B^TX_2\bar{B}$.
\end{lemma}

We can now describe the Sylow $p$-groups of the classical groups.
But first we define the matrix
\begin{equation}\label{eq:Jm}
J_n:=\left(
\begin{array}{cccc}
0& \cdots & 0 & 1\\
\vdots& \iddots & \iddots& 0\\
0 & 1& \iddots & \vdots \\
1& 0 & \cdots & 0
\end{array}
\right).
\end{equation}
In Section \ref{sec:proofsylows} we will give proofs for the following Lemmas.

\begin{lemma}\label{le:sylowGUandSp}
The following holds:
\begin{enumerate}
    \item The group $\mathfrak G_{J_{m-1},J_2}^{+}$ is a Sylow $p$-subgroup for $GU(2m,q^2)$.

     \item Let $G$ consist of the elements of $\mathfrak G_{J_{m},1}^{+}$
with $F$ the $1\times 1$ identity matrix. Then $G$ is a Sylow $p$-subgroup for $GU(2m+1,q^2)$.

    \item The group $\mathfrak G_{J_{m-1},X_2}^{-}$ is a Sylow $p$-subgroup for $Sp(2m,q)$ with
$X_{2}=\left(
\begin{tabular}[h]{cc}
$0$ & $1$\\
$-1$ &$0$
\end{tabular}
\right)$.
\end{enumerate}
\end{lemma}

We consider separately the orthogonal groups in odd and in even characteristic.

\begin{lemma}\label{le:sylowOrthogonaloddq}
Assume that $q$ is odd. Then:
\begin{enumerate}
    \item The group $\mathfrak G_{J_{m},2}^{+}$ is a Sylow $p$-subgroup for $O(2m+1,q)$.

    \item The group $\mathfrak G_{J_{m-1},J_2}^{+}$ is a Sylow $p$-subgroup for $O^+(2m,q)$.

    \item The group $\mathfrak G_{J_{m},X_2}^{+}$ is a Sylow $p$-subgroup for $O^-(2m+2,q)$ with
$X_{2}=\left(
\begin{tabular}[h]{cc}
$a$ & $1$\\
$1$ &$2a$
\end{tabular}
\right)$. Here $a$ is such that $X^2+X+a$ is irreducible in $\mathbb{F}_q[X]$.
\end{enumerate}
\end{lemma}

Consider the matrices
\begin{equation}\label{matrixLL1}
\begin{tabular}[h]{ccc}
$L:=\left(
\begin{tabular}[h]{c|c|c}
$I_{m-1}$ & $0$& $0$ \\[-0.4 cm]
& \\
\hline
& \\[-0.4 cm]
$0$ &$J_2$ & $0$
\\[-0.4 cm]
& \\
\hline
& \\[-0.4 cm]
$0$ & $0$&$I_{m-1}$
\end{tabular}
\right)$ &\qquad \qquad& $L_1:=\left(
\begin{tabular}[h]{c|c|c}
$I_{m}$ & $0$& $0$ \\[-0.4 cm]
& \\
\hline
& \\[-0.4 cm]
$0$ &$J_2^\prime$ & $0$
\\[-0.4 cm]
& \\
\hline
& \\[-0.4 cm]
$0$ &$0$&$I_{m}$
\end{tabular}
\right)$
\end{tabular}
\end{equation}
where $I$ denotes the identity matrix and
$J_2^\prime:=\left(
\begin{tabular}[h]{cc}
$1$ & $1$ \\
$0$ &$1$
\end{tabular}
\right)
$.

\begin{lemma}\label{le:sylowOrthogonalevenq}
Take $q$ even.
\begin{enumerate}
\item  Let $G$ consist of the elements of $\mathfrak G_{J_{m},0}^{+}$
with $F$ the $1\times 1$ identity matrix and $s_{ii}=b_{1i}^2$ for $i=1,\ldots,m$ (in notation of
Section \ref{sec:sylowdescription}). Then $G$ is a Sylow $p$-subgroup for $O(2m+1,q)$.
\item Let $G_1$ be the group formed by the elements of $\mathfrak G_{J_{m-1},J_2}^{+}$ that satisfy $s_{ii}=b_{1i}b_{2i}$ for $i=1,\ldots,m-1$. Then the group generated by $G_1$ and $L$ is a Sylow $p$-subgroup for $O^+(2m,q)$.
\item Let $G_1$ be the group formed by the elements of $\mathfrak G_{J_{m},J_2}^{+}$ that satisfy $s_{ii}=b_{1i}^2+b_{1i}b_{2i}+b_{2i}^2$ for $i=1,\ldots,m$. Then the group generated by $G_1$ and $L_1$ is a Sylow $p$-subgroup for $O^-(2m+2,q)$.
\end{enumerate}
\end{lemma}

\section{Invariant Fields}\label{sec:invfields}
Let $\mathbb{F}$ be a field, $V=\mathbb{F}^n$ and let
$(x_1,x_2,\ldots,x_n)$ be the dual basis of the ordered standard basis
$(e_1,\cdots,e_n)$ of $V$. We will write elements of $V$ as columns vectors of
the form $v:=(\alpha_1,\cdots,\alpha_n)^T$.
\\[1mm]
If $\mathbb{F}=\mathbb{F}_q$,  the generators of $\mathbb{F}_q[V]^{GL(V)}$ can be
defined as the coefficients of the polynomial
\begin{equation}\label{eq:Fn(X)}
F_{n,q}(X):=\prod_{u\in V^*}(X-u)=X^{q^n}+\sum_{i=0}^{n-1}(-1)^{n-i}c_iX^{q^i} \in \mathbb{K}[X],
\end{equation}
where $\mathbb{K}$ is a field containing $\mathbb{F}(V)$.

\begin{lemma}\label{lem:Fn(X)ind}
Let $U:=\langle x_1,\cdots,x_{n-1}\rangle_{\mathbb{F}_q}$. Then we have
\[
F_{n,q}(X)=F_{n-1,q}(X)^q-F_{n-1,q}(x_n)^{q-1}F_{n-1,q}(X)
\]
where $F_{n-1,q}(X)=\prod_{u\in U}(X-u)$.
\end{lemma}
\begin{proof}
First, we note that the polynomial $F_{n,q}(X)$ is $\mathbb{F}_q$-linear. Hence
\begin{eqnarray}
\nonumber F_{n,q}(X)&=&\prod_{f\in V^*}(X-f)=\prod_{a\in \mathbb{F}_q}\prod_{g\in U}(X-ax_n-g)\\
\nonumber  &=&\prod_{a\in \mathbb{F}_q}F_{n-1,q}(X-ax_n)=\prod_{a\in \mathbb{F}_q}(F_{n-1,q}(X)-aF_{n-1,q}(x_n))\\
\nonumber  &=&F_{n-1,q}(X)^q-F_{n-1,q}(x_n)^{q-1}F_{n-1,q}(X).
\end{eqnarray}
This finishes the proof.
\end{proof}
Obviously we have $F_{0,q}(X)=X$. Let $\mathbb{F}$ be either the finite field  $\mathbb{F}_q$ or $\mathbb{F}_{q^2}$ and denote by $r$ the number of elements of $\mathbb{F}$.
We define a sequence of endomorphisms $\psi_l$ of $\mathbb{F}$-algebras from
$A:=\mathbb{F}[x_1,\ldots,x_n]$ to itself by $\psi_l:\ A \longrightarrow A$,
$x_i\mapsto F_{l,r}(x_i).$
Note that $\psi_0$ is the identity map on $A$, $\psi_1(x_1)=0$ and $\psi_1(x_2)=x_2^r-x_1^{r-1}x_2$ is the orbit product of $x_2$ under the action of $U(n,\mathbb{F})$.

\begin{proposition}\label{prop:psil}
For every endomorphism $\psi_l$ the following hold:
\begin{enumerate}
    \item $\psi_l(x_k)=0$ for all $1\leq k\leq l$;
    \item $\psi_l(x_{l+1})$ is the orbit product of $x_{l+1}$ under the action of  $U(n,\mathbb{F})$,
    hence an invariant for that group. 
    \item $\psi_l(f)=(\psi_{l-1}(f))^r-\psi_{l-1}(x_l)^{r-1}\psi_{l-1}(f)$ for every homogeneous polynomial $f$ in degree $1$, i.e, $\mathbb{F}$-linear combinations of the $x_i$'s;
    \item for every $g\in U(n,\mathbb{F})$ we have $g\circ\psi_l=\psi_l\circ g$.
\end{enumerate}
\end{proposition}
\begin{proof}
(1).:\ We prove this by induction on $l$. For $l=1$ we have seen that $\psi_1(x_1)=0$. Now we assume that the statement is true for $l$ and let $k\leq l+1$. Then
$\psi_{l+1}(x_k)=\psi_l(x_k)^r-\psi_l(x_{l+1})^{r-1}\psi_l(x_k),$
which is zero for $k\leq l$ by the induction hypothesis. For $k=l+1$ we get $\psi_{l+1}(x_{l+1})=0$ immediately. (2).:\ By definition $\psi_l(x_{l+1})=F_{l,r}(x_{l+1})$ and the statement from Lemma \ref{lem:Fn(X)ind}.
(3).:\ Note that the endomorphisms $\psi_l$ as well as multiplication by the fixed element $\psi_{l-1}(x_l)^{r-1}$ and $(\ )^r$ are $\mathbb{F}$-linear operators. Since the formula is true for each $x_i$ by definition, the result follows.
(4).:\ Here it suffices to show that $(g\circ\psi_l)(x_i)=(\psi_l\circ g)(x_i)$ for all $i=1,2,\ldots,n$. Again, we use induction on $l$. For $l=0$ the result follows immediately since $\psi_0$ is the identity map.
We assume that the result holds for $l$. Then
\begin{eqnarray}
\nonumber (g\circ\psi_{l+1})(x_i)&=&g(\psi_{l+1}(x_i))=g(\psi_{l}(x_i)^r-\psi_{l}(x_{l+1})^{r-1}\psi_{l}(x_i))\\
\nonumber &=&(g(\psi_{l}(x_i)))^r-(g(\psi_{l}(x_{l+1})))^{r-1}(g(\psi_{l}(x_i)))\\
\nonumber &=& \psi_{l}(g(x_i))^r-\psi_{l}(g(x_{l+1}))^{r-1}\psi_{l}(g(x_i))
\end{eqnarray}
where have used the induction hypothesis. It follows from 2 that $\psi_l(x_{l+1})$ is invariant and therefore $\psi_l(g(x_{l+1}))=\psi_l(x_{l+1})$. Hence
$$(g\circ\psi_{l+1})(x_i)=\psi_{l}(g(x_i))^r-\psi_{l}(x_{l+1})^{r-1}\psi_{l}(g(x_i))
=(\psi_{l+1}\circ g)(x_i)$$
and this finishes the proof.
\end{proof}

We consider the following families of polynomials in $\mathbb{F}[x_1,\ldots,x_n]$.
We use two parameters: $j\in\{-1,1\}$ and $\lambda\in\mathbb{F}$. Let  $m=\displaystyle\frac{n}{2}$ or $m=\displaystyle\frac{n-1}{2}$ if $n$ is even or odd, respectively. Now define
\begin{itemize}
    \item $\Omega_{0,1}=\sum_{i=1}^mx_{n-i+1}x_i$ and $\Omega_{0,-1}=0$;

    \item $\Omega_{s,j}=\sum_{i=1}^m(x_{n-i+1}^{r^s}x_i+jx_{n-i+1}x_i^{r^s})$ for $s\geq 1$;

    \item $\Gamma_{0,\lambda,}=\Omega_{0,1}+x_{m+1}^{2}+\lambda\,x_{m+2}^{2}$;

    \item $\Gamma_{s,\lambda,}=\Omega_{s,1}+2\cdot 1_\mathbb{F}(x_{m+1}^{r^{s}+1}+\lambda\,x_{m+2}^{r^{s}+1})$ for $s\geq 1$;
    \item $\Lambda_{s,\lambda}=\sum_{i=1}^m(x_{n-i+1}^{q^{2s-1}}x_i+x_{n-i+1}x_i^{q^{2s-1}})+\lambda\, x_{m+1}^{q^{2s-1}+1}$ for $s\geq 1$ and $\mathbb{F}=\mathbb{F}_{q^2}$.
\end{itemize}
We will apply the Steenrod operations to these polynomials (see the introduction for its definition). Here we take $\zeta=-1$ and we denote $\mathcal{P}(-1)$ by $\mathcal{P}^\bullet$. Hence $\mathcal{P}^\bullet:A\longrightarrow A$ is the $\mathbb{F}$-algebra homomorphism  given by $\mathcal{P}^\bullet(x_i)=x_i-x_i^r$. Also,
$$\mathcal{P}^\bullet(f)=\mathcal{P}^0(f)-\mathcal{P}^1(f)+\mathcal{P}^2(f)-\mathcal{P}^3(f)+\cdots$$
where $\mathcal{P}^i(f)$ is the $i$-th Steenrod operation on $f$.
The next Lemmas will be proved in Section 5.

\begin{lemma}\label{cor:StOpOmegaGammaLamba}
The Steenrod operations on the polynomials $\Omega_{s,j}$, $\Gamma_{s,\lambda}$ and $\Lambda_{s,\lambda}$ are given by:
\begin{enumerate}
    \item $\mathcal{P}^1(\Omega_{0,1})=\Omega_{1,1}$, $\mathcal{P}^1(\Gamma_{0,\lambda})=\Gamma_{1,\lambda}$ and $\mathcal{P}^1(\Lambda_{1,\lambda})=\Lambda_{1,\lambda}^q$;
    \item $\mathcal{P}^1(\Omega_{1,1})=2\Omega_{0,1}^r$, $\mathcal{P}^1(\Omega_{s,j})=\Omega_{s-1,j}^r$ for $s\geq 2$, \\  $\mathcal{P}^1(\Gamma_{s,\lambda})=\Gamma_{s-1,\lambda}^r$ for $s\geq 1$ and $\mathcal{P}^1(\Lambda_{s,\lambda})=\Lambda_{s,\lambda}^{q^2}$ for $s\geq 2$;
    \item $\mathcal{P}^{r^s}(\Omega_{s,j})=\Omega_{s+1,j}$, $\mathcal{P}^{r^s}(\Gamma_{s,\lambda})=\Gamma_{s+1,\lambda}$ and $\mathcal{P}^{q^{2s-1}}(\Lambda_{s,\lambda})=\Lambda_{s+1,\lambda}$ for $s\geq 1$;
    \item $\mathcal{P}^{r^s+1}(\Omega_{s,j})=\Omega_{s,j}^r$, $\mathcal{P}^{r^s+1}(\Gamma_{s,\lambda})=\Gamma_{s,\lambda}^r$ for $s \geq 0$ and \\
$\mathcal{P}^{q^{2s-1}+1}(\Lambda_{s,\lambda})=\Lambda_{s,\lambda}^{q^2} $ for $s\geq 1$;
    \item $\mathcal{P}^{i}(\Omega_{s,j})=0$, $\mathcal{P}^{i}(\Gamma_{s,\lambda})=0$ and $\mathcal{P}^{i}(\Lambda_{s,\lambda})=0$, otherwise.
\end{enumerate}
\end{lemma}

\begin{lemma}\label{cor:psilOmegaGammaDeltaalgebramemb}
Assume by convention that $\Omega_{s,j}=0$, $\Gamma_{s,\lambda}=0$ for $s<0$ and $\Lambda_{s}=0$ for $s\leq 0$. Then:
\begin{enumerate}
    \item $\psi_l(\Omega_{s,j})\in \mathbb{F}[x_1,\psi_1(x_2),\ldots,\psi_{l-1}(x_l),\Omega_{s-1,j},\Omega_{s,j},\Omega_{s+1,j},\ldots,\Omega_{s+l,j}]$;
    \item $\psi_l(\Gamma_{s,\lambda})\in \mathbb{F}[x_1,\psi_1(x_2),\ldots,\psi_{l-1}(x_l),\Gamma_{s-1,\lambda},\Gamma_{s,\lambda},\Gamma_{s+1,\lambda},\ldots,\Gamma_{s+l,\lambda}]$;
    \item $\psi_l(\Lambda_{s,\lambda})\in \mathbb{F}[x_1,\psi_1(x_2),\ldots,\psi_{l-1}(x_l),\Lambda_{s-1,\lambda},\Lambda_{s,\lambda},\Lambda_{s+1,\lambda},\ldots,\Lambda_{s+l,\lambda}]$.
\end{enumerate}
\end{lemma}

\begin{proposition}\label{prop:degreesphijOmegaGammaDelta}
For every $l\geq 0$ and $s>0$, the polynomials $\psi_l(\Omega_{0,1})$, $\psi_l(\Omega_{s,j})$, $\psi_l(\Gamma_{0,\lambda})$, $\psi_l(\Gamma_{s,\lambda})$ and $\psi_l(\Lambda_{s,\lambda})$ belong to $\mathbb{F}[x_{1},\ldots,x_{n-l}]$. Moreover, for $0\leq l\leq m-1$, their degree in the variable $x_{n-l}$ is:
\begin{enumerate}
    \item  $r^l$ for $\psi_l(\Omega_{0,1})$ and $\psi_l(\Gamma_{0,\lambda})$;
    \item  $r^{l+s}$ for $\psi_l(\Omega_{s,j})$ and $\psi_l(\Gamma_{s,\lambda})$;
    \item  $q^{2l+2s-1}$ for $\psi_l(\Lambda_{s,\lambda})$.
\end{enumerate}
\end{proposition}
\begin{proof}
Since $\psi_l(x_i)=0$ for all $i\leq l$, it is easy to see that $\psi_l(\Omega_{0,1}), \psi_l(\Omega_{s,j})$, $\psi_l(\Gamma_{0,\lambda})$, $\psi_l(\Gamma_{s,\lambda})$, $\psi_l(\Lambda_{s,\lambda})\in \mathbb{F}[x_{1},\ldots,x_{n-l}]$.
\\
By definition $\psi_l(x_i)=F_{l,r}(x_i)$ and it can be easily proven by induction on $l$ that $F_{l,r}(x_i)\in \mathbb{F}[x_{1},\ldots,x_{i}]$ with degree $r^l$ in $x_i$.
Since
$$\psi_l(\Omega_{0,1})=\sum_{i=l+1}^{m}\psi_l(x_{n-i+1})\psi_l(x_{i})$$
we conclude that $\psi_l(\Omega_{0,1})$ and, similarly,
$\psi_l(\Gamma_{0,\lambda})$ have degree equal to $r^l$ in $x_{n-l}$ for $0\leq l\leq m-1$.
For $\Omega_{s,j}$,  we have
$$\psi_l(\Omega_{s,j})=\sum_{i=l+1}^{m}\psi_l(x_{n-i+1})^{r^s}\psi_l(x_{i})+j\psi_l(x_{n-i+1})\psi_l(x_{i})^{r^s}$$
and therefore $\psi_l(\Omega_{s,j})$ has degree $r^{l+s}$ in $x_{n-l}$. Similar arguments give us the results for $\psi_l(\Gamma_{s,\lambda})$ and $\psi_l(\Lambda_{s,\lambda})$.
\end{proof}
We now introduce two subgroups of $U(2t+d,\mathbb{F})$.
Let $H^+$ be the set of matrices
$$\left(
\begin{tabular}{c|c|c}
$I_t$ & $0$ & $0$ \\
\hline
$0$ & $I_d$ & $0$\\
\hline
$C^+$ & $0$ & $I_t$
\end{tabular}
\right)$$
where $I_t$ and $I_d$ are the $t\times t$ and $d\times d$ identity matrices, respectively; and $C^+$ is any $t\times t$ matrix with entries in $\mathbb{F}$ such that
$$c_{i,j}^+=\overline{c}_{t-j+1,t-i+1}^+ \qquad \textnormal{ for all } i \textnormal{ and } j.$$
It is not hard to check that $H^+$ is an abelian subgroup of $U(2t+d,\mathbb{F})$.
If in the elements of $H^+$ we replace the matrix $C^+$ by a matrix $C^-$, of the same dimension, such that
$$c_{i,j}^-=-\overline{c}_{t-j+1,t-i+1}^- \qquad \textnormal{ for all } i \textnormal{ and } j$$
we obtain another abelian subgroup of $U(2t+d,\mathbb{F})$. We denote it by $H^-$.
Let $P[k]$ denote the polynomial ring $\mathbb{F}[x_1,\ldots,x_{t+d},x_{t+d+1},\ldots,x_{t+d+k}]$.

\begin{proposition}\label{prop:minimaldegrees}
Let $k\in \{1,\ldots,t\}$. Then the minimal degree in $x_{t+d+k}$ of:
\begin{enumerate}
    \item  a polynomial in $R[t+d+k]^{H^+}$ is greater than or equal to
$$\left\{
\begin{array}{ll}
q^{2(t-k)+1}& \textnormal{ if } \mathbb{F}=\mathbb{F}_{q^2}\\
q^{t-k+1}  & \textnormal{ if } \mathbb{F}=\mathbb{F}_{q}
\end{array}
\right..$$
    \item  a polynomial in $R[t+d+k]^{H^-}$ is greater than or equal to
$$\left\{
\begin{array}{ll}
q^{2(t-k)+1}& \textnormal{ if } \mathbb{F}=\mathbb{F}_{q^2}\\
q^{t-k}& \textnormal{ if } \mathbb{F}=\mathbb{F}_{q} \textnormal{ and } q \textnormal{ odd }\\
q^{t-k+1} & \textnormal{ if } \mathbb{F}=\mathbb{F}_{q} \textnormal{ and } q \textnormal{ even }
\end{array}
\right..$$
\end{enumerate}
\end{proposition}

\begin{remark}\label{re:H+-diagonalzero}
Assume that $\mathbb{F}=\mathbb{F}_{q}$ with $q$ even and that in the elements of $H^-$ the matrices $C^-$ also satisfy $c_{i,t-i+1}^-=0$. Then it follows from the proof of Proposition \ref{pro:Lk-} that the minimal degree in $x_{t+d+k}$ of a polynomial in $R[t+d+k]^{H^-}$ will be greater than or equal to $q^{t-k}$.
\end{remark}

The invariant rings for the following two subgroups of $U(n,\mathbb{F})$ will be important in the subsequent subsections. Let
\begin{itemize}
    \item $U_1$ be the set of elements $u \in U(n,\mathbb{F})$ such that $u(x_j)=x_j+\sum_{k=1}^{j-1}a_{jk}x_k$, for $1\leq j\leq n-1$ and  $u(x_n)=x_n+\sum_{k=1}^{n-2}a_{nk}x_k$;

    \item $U_2$ be the set of elements $u \in U(n,\mathbb{F}_{q^2})$ such that $u(x_j)=x_j+\sum_{k=1}^{j-1}a_{jk}x_k$, for $1\leq j\leq n-1$ and  $u(x_n)=x_n+bx_{n-1}+\sum_{k=1}^{n-2}a_{nk}x_k$ with $b+\bar{b}=0$.
\end{itemize}

\begin{lemma}\label{le:phijingeneral}
Let $U_1$ and $U_2$ be the groups defined above. For each $j\in\{1,\ldots,n\}$ and $k\in\{1,2\}$ we have
$$\mathbb{F}[x_1,x_2,\ldots,x_j]^{U_k}=\mathbb{F}[x_1,N(x_2),\ldots,N(x_j)]$$
where $N(x_i)$ is the orbit product of $x_i$ for $i\leq j$. Furthermore, the degree in $x_j$ of $N(x_j)$ is minimal among the elements in $\mathbb{F}[x_1,x_2,\ldots,x_j]^{U_k}$.
\end{lemma}

\subsection{The Invariant Field of a Sylow $p$-subgroup of $GU(2m,q^2)$}
Here $\mathbb{F}=\mathbb{F}_{q^2}$ and $n=2m$. Let $G$ denote the Sylow $p$-subgroup of $GU(2m,q^2)$ given in Lemma \ref{le:sylowGUandSp}.
First, we introduce a family of polynomials which we shall prove to be invariants under the action of $G$. For $k\geq 1$ define
$$h_k:=\Lambda_{k,0}.$$
Thus $h_1=\sum_{i=1}^m(x_{2m-i+1}^{q}x_i+x_{2m-i+1}x_i^{q})$.
\begin{lemma}\label{le:hkinvariantsGUeven}
For all $k\geq 1$ the polynomials $h_k$ belong to $\mathbb{F}_{q^2}[V]^{GU(2m,q^2)}$.
\end{lemma}
\begin{proof}
From Lemma \ref{cor:StOpOmegaGammaLamba} we get $h_k=\mathcal{P}^{q^{2k-3}}(h_{k-1})$ for $k>1$, where $\mathcal{P}^{q^{2k-3}}$ is the $q^{2k-3}$-th Steenrod operation. Hence it is enough to prove that $h_1$ is an invariant polynomial. In order to use polynomial functions rather than
polynomials in ${\rm Sym}(V^*)$ we take $v=(\alpha_1,\cdots,\alpha_n)\in\bar{V}=\bar{\mathbb{F}}_{q^2}^n$, where $\bar{\mathbb{F}}_{q^2}$ is the algebraic closure of $\mathbb{F}_{q^2}$.
Thus $h_1(v)=\sum_{i=1}^m(\alpha_{2m-i+1}^{q}\alpha_i+\alpha_{2m-i+1}\alpha_i^{q})=$
$h_1(v)=v^TJ_{2m}\bar{v}$.
Now, let $M \in GU(2m,q^2)$, then $(M.h_1)(v)=h_1(M^{-1}v)=$ $(M^{-1}v)^TJ_{2m}\overline{M^{-1}v}=$ $$v^T(M^{-1})^TJ_{2m}\bar{M}^{-1}\bar{v}=v^TJ_{2m}\bar{v}=h_1(v),$$
where we have used the definition of $GU(2m,q^2)$.
Hence $M.h_1=h_1$.
\end{proof}

\begin{theorem}\label{the:invariantfieldGUeven}
Let $G$ be the Sylow $p$-group of $GU(2m,q^2)$ as described in Lemma \ref{le:sylowGUandSp}.
The invariant field $\mathbb{F}_{q^2}(V)^G$ is generated by the polynomials $N(x_j)$, with $j=1,\ldots,m+1$, and the polynomials $h_k$, with $k=1,\ldots,m-1$,i.e.,
\[
\mathbb{F}_{q^2}(V)^G=\mathbb{F}_{q^2}(x_1,N(x_2),\ldots,N(x_{m+1}),h_1,\ldots,h_{m-1}).
\]
\end{theorem}
\begin{proof}
We shall use Theorem \ref{the:invfieldgen} to get the result.
We start by noting that the matrices $F$ in the elements of $G$ look like
$$\left(
\begin{tabular}[h]{cc}
$1$ & $0$ \\
$c$ &$1$
\end{tabular}
\right)$$
where $c$ is an element in $\mathbb{F}_{q^2}$ satisfying $c+\bar{c}=0$. Hence $G$ acts on $R[m+1]$ in the same way as the group $U_2$ in Lemma \ref{le:phijingeneral} and we obtain that $R[j]^G=R[j]^{U_2}$ for each $j\in\{1,\ldots,m+1\}$. It also follows from Lemma \ref{le:phijingeneral} that $N(x_j)$ is an element in $R[j]^G$ of minimal degree in $x_j$. Therefore, for each $j\in\{1,\ldots,m+1\}$ we choose $\phi_j=N(x_j)$.

Now, if we consider all the elements of $G$ for which $A$ and $F$ are the identity matrices and $B$ the zero matrix, then we obtain an abelian subgroup $H$ of $G$ whose elements are
$$
\left(
\begin{tabular}[h]{c|c|c}
$I_{m-1}$ &$0$ & $0$ \\[-0.4 cm]
& \\
\hline
& \\[-0.4 cm]
$0$ &$I_2$ & $0$
\\[-0.4 cm]
& \\
\hline
& \\[-0.4 cm]
$J_{m-1}\bar{S}$ &$0$ &$I_{m-1}$
\end{tabular}
\right)
$$
with $S\in M(m-1,q^2)$ such that $S+\bar{S}^T=0$ and $J_{m-1}$ is the matrix given by (\ref{eq:Jm}) in
Section \ref{sec:sylowdescription}. Let $C=J_{m-1}\bar{S}$. Note that the multiplication by $J_{m-1}$ swaps the rows $i$ and $(m-1)-i+1=m-i$ of $\bar{S}$ for all $i$. Thus, since $S+\bar{S}^T=0$ we obtain
$c_{i,j}=\bar{s}_{m-i,j}=-s_{j,m-i}=-\bar{c}_{m-j,m-i}.$
Now, assume that $C$ is a $(m-1)\times (m-1)$ matrix with entries in $\mathbb{F}_{q^2}$ such that $c_{i,j}=-\overline{c}_{(m-1)-j+1,(m-1)-i+1}=-\overline{c}_{m-j,m-i}.$
By taking $S=J_{m-1}\bar{C}$  we get
$s_{i,j}=-\overline{c}_{m-i,j}=-c_{m-j,i}=-\overline{s}_{j,i}$
and therefore $S+\bar{S}^T=0$. Hence $H$ is the subgroup $H^-$ with $t=m-1$ and $d=2$.
Let $k\in\{1,\ldots,m-1\}$.
Since $R[m+1+k]^G\subset R[m+1+k]^{H^-},$
applying Proposition \ref{prop:minimaldegrees} we obtain that the minimal degree in $x_{m+1+k}$ of an element in $R[m+1+k]^G$ is greater than or equal to $q^{2(m-k)-1}$. By Proposition \ref{prop:degreesphijOmegaGammaDelta} this is the $x_{m+1+k}$-degree of
$\psi_{m-1-k}(\Lambda_{1,0})=\psi_{m-1-k}(h_1).$
Now, we have $\psi_{m-1-k}(h_1)\in R[m+1+k]^G$ by Proposition \ref{prop:psil}-4 and Lemma \ref{le:hkinvariantsGUeven}. Hence we can take $\phi_{m+1+k}=\psi_{m-1-k}(h_1)$. 
It follows from Theorem \ref{the:invfieldgen} that
$$\mathbb{F}_{q^2}(V)^G=\mathbb{F}_{q^2}(x_1,N(x_2),\ldots,N(x_{m+1}),\psi_{m-2}(h_1),\ldots,\psi_1(h_1),h_1).$$
Applying Lemma \ref{cor:psilOmegaGammaDeltaalgebramemb} and Proposition \ref{prop:psil}-2 we get for each $k<m-1$
$$
\psi_{m-1-k}(h_1)\in \mathbb{F}_{q^2}[x_1,N(x_2),\ldots,N(x_{m-1-k}),h_1,h_2,\ldots,h_{m-k}].
$$
Hence
$$\mathbb{F}_{q^2}(V)^G=\mathbb{F}_{q^2}(x_1,N(x_2),\ldots,N(x_{m+1}),h_{m-1},\ldots,h_{1})$$
and this finishes the proof.
\end{proof}

\subsection{The Invariant Field of a Sylow $p$-subgroup of $GU(2m+1,q^2)$}
Here $\mathbb{F}=\mathbb{F}_{q^2}$ and $n=2m+1$. We consider the following family of polynomials: for $k\geq 1$ let
$$h_k:=\Lambda_{k,1}.$$
Thus $h_1=\sum_{i=1}^m(x_{2m+1-i+1}^{q}x_i+x_{2m+1-i+1}x_i^{q})+x_{m+1}^{q+1}$.

\begin{lemma}\label{le:hkinvGUodd}
For all $k\geq 1$ the  $h_k$ belong to $\mathbb{F}_{q^2}[V]^{GU(2m+1,q^2)}$.
\end{lemma}
\begin{proof}
From Lemma \ref{cor:StOpOmegaGammaLamba} we get $h_k=\mathcal{P}^{q^{2k-3}}(h_{k-1})$ for $k>1$. Hence it suffices to prove that $h_1$ is invariant. This now is entirely analogous to
the arguments in the proof of Lemma \ref{le:hkinvariantsGUeven}.
\end{proof}

\begin{theorem}\label{the:invariantfieldGUodd}
Let $G$ denote the Sylow $p$-subgroup of $GU(2m+1,q^2)$ given by Lemma \ref{le:sylowGUandSp}.
The invariant field $\mathbb{F}_{q^2}(V)^G$ is generated by the polynomials $N(x_j)$, with $j=1,\ldots,m+1$, and the polynomials $h_k$, with $k=1,\ldots,m$, i.e.,
\[
\mathbb{F}_{q^2}(V)^G=\mathbb{F}_{q^2}(x_1,N(x_2),\ldots,N(x_{m+1}),h_1,\ldots,h_{m}).
\]
\end{theorem}
\begin{proof}
In this case $G$ acts on $R[m+1]$ like the group $U(m+1,q^2)$. Therefore, for each $j\in\{1,\ldots,m+1\}$, the degree in $x_j$ of $N(x_j)\in R[j]^G$ is minimal and we can take $\phi_j=N(x_j)$.
If we consider the elements of $G$ for which $A$ is the identity matrix and $v$ is the zero vector, then we obtain an abelian subgroup $H$.
Similarly to what was done in the proof of Theorem \ref{the:invariantfieldGUeven}, we can show that $H$ is the group $H^-$ with $t=m$ and $d=1$.
Let $k\in\{0,\ldots,m\}$. Since
$R[m+1+k]^G\subset R[m+1+k]^{H^-}$
and using Proposition \ref{prop:minimaldegrees}, we can see that the minimal degree in $x_{m+1+k}$ of a polynomial in $R[m+1+k]^G$ is greater than or equal to $q^{2(m-k)+1}$. By Proposition \ref{prop:degreesphijOmegaGammaDelta}, this is the $x_{m+1+k}$-degree of
$$\psi_{m-k}(\Lambda_{1,1})=\psi_{m-k}(h_1)\in R[m+1+k].$$
According to Lemma \ref{le:hkinvGUodd} and Proposition \ref{prop:psil}-4, we have that $\psi_{m-k}(h_1)$ is invariant for $G$  and therefore we can take $\phi_{m+1+k}=\psi_{m-k}(h_1)$.
Now, from Lemma \ref{cor:psilOmegaGammaDeltaalgebramemb}, Proposition \ref{prop:psil}-2, it follows that for $k<m$, 
$\psi_{m-k}(h_1)\in \mathbb{F}_{q^2}[x_1,N(x_2),\ldots,N(x_{m-k}),h_1,h_2,\ldots,h_{m+1-k}].$
Finally, applying Theorem \ref{the:invfieldgen} we conclude that
$$\mathbb{F}_{q^2}(V)^G=\mathbb{F}_{q^2}(x_1,N(x_2),\ldots,N(x_{m+1}),h_{m},\ldots,h_{1})$$
and this finishes the proof.
\end{proof}

\subsection{The Invariant Field of a Sylow $p$-subgroup of $Sp(2m,q)$}
Here $\mathbb{F}=\mathbb{F}_q$ and $n=2m$.
Now, for each $k\geq 1$ let $h_k:=\Omega_{k,-1}.$
Thus $h_1=\sum_{i=1}^m(x_{2m-i+1}^{q}x_i-x_{2m-i+1}x_i^{q})$.
\begin{lemma}\label{hkinvariantSp}
For all $k\geq 1$ the polynomials $h_k$ belong to $\mathbb{F}_{q}[V]^{Sp(2m,q)}$.
\end{lemma}
\begin{proof}
By Lemma \ref{cor:StOpOmegaGammaLamba}, $h_k=\mathcal{P}^{q^{k-1}}(h_{k-1})$ for $k>1$ and so it is enough to prove that $h_1$ is an invariant polynomial, which is done in the same way as
in the proof of Lemma \ref{le:hkinvariantsGUeven}.
\end{proof}

\begin{theorem}\label{the:invfieldsSp}
Let $G$ be the Sylow $p$-subgroup of $Sp(2m,q)$ given by Lemma \ref{le:sylowGUandSp}.
The invariant field $\mathbb{F}_{q}(V)^G$ is generated by the polynomials $N(x_i)$, with $i=1,\ldots,m+1$, and the polynomials $h_k$, with $k=1,\ldots,m-1$, i.e.,
\[
\mathbb{F}_{q}(V)^G=\mathbb{F}_{q}(x_1,N(x_2),\ldots,N(x_{m+1}),h_1,\ldots,h_{m-1}).
\]
\end{theorem}
\begin{proof}
By choosing the elements of $G$ for which $A$ and $F$ are the identity matrices and $B$ is the zero matrix, we obtain an abelian subgroup $H$ of $G$ with elements
$$
\left(
\begin{tabular}[h]{c|c|c}
$I_{m-1}$ &$0$ & $0$ \\[-0.4 cm]
& \\
\hline
& \\[-0.4 cm]
$0$ &$I_2$ & $0$
\\[-0.4 cm]
& \\
\hline
& \\[-0.4 cm]
$J_{m-1}S$ &$0$ &$I_{m-1}$
\end{tabular}
\right)
$$
where $S\in M(m-1,q)$ is such that $S-S^T=0$.
It is easy to check that if $C=J_{m-1}S$ then $c_{i,j}=c_{m-j,m-i}$. Also if $C$ is any matrix with entries in $\mathbb{F}_{q}$ satisfying $c_{i,j}=c_{m-j,m-i}$ then $S=J_{m-1}C$ satisfies $S-S^T=0$. Hence $H$ is the group $H^+$ with $t=m-1$ and $d=2$.
Now, let $k\in\{1,\ldots,m-1\}$. Since
$R[m-1-k]^G\subset R[m-1-k]^{H^+}$
the minimal degree in $x_{m+1+k}$ of a polynomial in $R[m+1+k]^G$ is, according to Proposition \ref{prop:minimaldegrees}, greater than or equal to $q^{m-k}$. We know from Proposition \ref{prop:degreesphijOmegaGammaDelta} that $q^{m-k}$ is actually the degree of
$\psi_{m-1-k}(\Omega_{1,-1})=\psi_{m-1-k}(h_1).$
We know that $\psi_l(h_1)$ is an invariant polynomial by Proposition \ref{prop:psil}-4 and Lemma \ref{hkinvariantSp} and
therefore, we take $\phi_{m+1+k}=\psi_{m-1-k}(h_1)$.
Also, it follows from Lemma \ref{cor:psilOmegaGammaDeltaalgebramemb} and Proposition \ref{prop:psil}-2 that for $k<m-1$
$$\psi_{m-1-k}(h_1)\in \mathbb{F}_q[x_1,N(x_2),\ldots,N(x_{m-1-k}),h_1,h_2,\ldots,h_{m-k}].$$
Finally for each $j\in\{1,\ldots,m+1\}$, we compute the polynomial $\phi_j$. We note that $G$ is acting on $R[m+1]$ in the same way as is the group $U(m+1,\mathbb{F}_{q})$. Hence $R[j]^G=R[j]^{U(m+1,\mathbb{F}_{q})}$ and therefore we can choose $\phi_j=N(x_j)$. Applying Theorem \ref{the:invfieldgen} we conclude that
$$\mathbb{F}_{q}(V)^G=\mathbb{F}_{q}(x_1,N(x_2),\ldots,N(x_{m+1}),h_{m-1},\ldots,h_{1})$$
and this finishes the proof.
\end{proof}

\subsection{The Invariant Field of a Sylow $p$-subgroup of $O^+(2m,q)$}
Let again $V=\mathbb{F}_q^n$ with $n=2m$. The orthogonal group $O^+(2m,q)$ is the group of invertible matrices that preserve the quadratic form
$$Q(v)=\sum_{i=1}^{m}\alpha_{2m-i+1}\alpha_i$$
with $v=\sum_{i=1}^m(\alpha_iu_i+\alpha_{2m-i+1}v_i)$ (see \cite{Taylor92}).
Now consider the following family of polynomials: for $k\geq 1$ define $h_k:=\Omega_{k-1,1}.$
In particular, $h_1=\sum_{i=1}^mx_{2m-i+1}x_i$.
The next lemma shows that $h_k$ is invariant under the action of $O^+(2m,q)$ for all $k$.

\begin{lemma}\label{le:hkinvO+}
For all $k\geq 1$ the polynomials $h_k$ belong to $\mathbb{F}_{q}[V]^{O^+(2m,q)}$.
\end{lemma}
\begin{proof}
We know from Lemma \ref{cor:StOpOmegaGammaLamba} that for $k>1$, $h_k$ is the $q^{k-2}$-th Steenrod operation of $h_{k-1}$ and therefore we just have to show that $h_1$ is invariant. This follows directly from the definition of the group $O^+(2m,q)$.
\end{proof}

We have to consider separately the cases when the characteristic of $\mathbb{F}_{q}$ is $2$ and when it is not.


\begin{theorem}\label{the:invfieldO+odd}
Let $q$ be odd and let $G$ be the Sylow $p$-subgroup of $O^+(2m,q)$ given by Lemma \ref{le:sylowOrthogonaloddq}. The invariant field $\mathbb{F}_{q}(V)^G$ is generated by the polynomials $N(x_i)$, with $i=1,\ldots,m+1$, and the polynomials $h_k$, with $k=1,\ldots,m-1$, i.e.,
\[
\mathbb{F}_{q}(V)^G=\mathbb{F}_{q}(x_1,N(x_2),\ldots,N(x_{m+1}),h_1,\ldots,h_{m-1}).
\]
\end{theorem}
\begin{proof}
First let us consider the abelian subgroup $H$ of $G$ obtained by taking the elements of $G$ for which the matrices $A$ and $B$ are equal to the identity and the zero matrix, respectively. Analogously to
the proof of Theorem \ref{the:invariantfieldGUeven}, we can easily show that $H=H^-$ with $t=m-1$ and $d=2$.
Let $k\in\{1,\ldots,m-1\}$. We proceed analogously to the proof of Theorem \ref{the:invfieldsSp}. Note that
$R[m+1+k]^G\subset R[m+1+k]^{H^-}$
and therefore it follows from Proposition \ref{prop:minimaldegrees} that the minimal degree in $x_{m+1+k}$ of a polynomial in $R[m+1+k]^G$ is greater than or equal to $q^{m-1-k}$.
According to Proposition \ref{prop:degreesphijOmegaGammaDelta} this is the $x_{m+1+k}$-degree of
$$\psi_{m-1-k}(\Omega_{0,1})=\psi_{m-1-k}(h_1).$$
Now, $\psi_{m-1-k}(h_1)$ is invariant by Lemma \ref{le:hkinvO+} and Proposition \ref{prop:psil}-4. 
Hence, we can take $\phi_{m+1+k}=\psi_{m-1-k}(h_1)$.
Applying Lemma \ref{cor:psilOmegaGammaDeltaalgebramemb} and Proposition \ref{prop:psil}-2 we see that for $k<m-1$,
$$\psi_{m-1-k}(h_1)\in \mathbb{F}_q[x_1,N(x_2),\ldots,N(x_{m-1-k}),h_1,h_2,\ldots,h_{m-k}].$$

Now we determine for each $j\in\{1,\ldots,m+1\}$, the polynomial $\phi_j$. By looking at how $G$ acts on  $R[m+1]$ we can see it is acting in the same way as the group $U_1$ in Lemma \ref{le:phijingeneral}. Hence we can choose $\phi_j=N(x_j)$ for  $j\in\{1,\ldots,m+1\}$. Applying Theorem \ref{the:invfieldgen} we conclude that
$$\mathbb{F}_{q}(V)^G=\mathbb{F}_{q}(x_1,N(x_2),\ldots,N(x_{m+1}),h_{m-1},\ldots,h_{1})$$
which finishes the proof.
\end{proof}

Finally, we assume that the characteristic of $\mathbb{F}_{q}$ is $2$. Consider the subgroup $G_1$ of $O^+(2m,q)$ given in Lemma \ref{le:sylowOrthogonalevenq}.

\begin{lemma}\label{le:invfieldO+evenG1}
The invariant field for $G_1$ is generated by the polynomials $N(x_i)$, with $i=1,\ldots,m+1$, and the polynomials $h_k$, with $k=1,\ldots,m-1$, i.e.,
\[
\mathbb{F}_{q}(V)^{G_1}=\mathbb{F}_{q}(x_1,N(x_2),\ldots,N(x_{m+1}),h_1,\ldots,h_{m-1}).
\]
\end{lemma}
\begin{proof}
First we would like to note that in the proof of Theorem \ref{the:invfieldO+odd} the only time we made use of the characteristic of $\mathbb{F}_{q}$ was when we applied Proposition \ref{prop:minimaldegrees}.
If we consider the elements of $G_1$ with $A$ equal to the identity matrix and $B$ the zero matrix, then we obtain an abelian subgroup $H_1$ with elements
$$
\left(
\begin{tabular}[h]{c|c|c}
$I_{m-1}$ &$0$ & $0$ \\[-0.4 cm]
& \\
\hline
& \\[-0.4 cm]
$0$ &$I_2$ & $0$
\\[-0.4 cm]
& \\
\hline
& \\[-0.4 cm]
$J_{m-1}S$ &$0$ &$I_{m-1}$
\end{tabular}
\right)
$$
where $S\in M(m-1,q)$ is such that $S+S^T=0$ and $s_{ii}=0$. Hence for $C=J_{m-1}S$ we also have $c_{i,m-i}=0$.
Now, we can use Remark \ref{re:H+-diagonalzero} instead of Proposition \ref{prop:minimaldegrees} to obtain the same conclusion as in the proof of Theorem \ref{the:invfieldO+odd} about the minimal degrees in $x_{m+1+k}$.
The rest of the proof is similar to the one of Theorem \ref{the:invfieldO+odd}.
\end{proof}

\begin{theorem}\label{the:invfieldO+even}
Let $G$ be the Sylow $p$-subgroup of $O^+(2m,q)$, with $q$ even, given by Lemma \ref{le:sylowOrthogonalevenq}. Then $\mathbb{F}_{q}(V)^{G}$ is generated by the polynomials
\[
x_1,N(x_2),\ldots,N(x_{m-1}),N(x_{m})+N(x_{m+1}),N(x_{m})N(x_{m+1}),h_1,\ldots,h_{m-1}.
\]
\end{theorem}
\begin{proof}
We showed  in the proof of Lemma \ref{le:invfieldO+evenG1} that $L$ normalises $G_1$. Hence $G_1$ is a normal subgroup of $G$ and $G/G_1=<L>$. We have
$$\mathbb{F}_{q}(V)^{G}=(\mathbb{F}_{q}(V)^{G_1})^{<L>}$$
and applying Lemma \ref{le:invfieldO+evenG1} we get
$$\mathbb{F}_{q}(V)^{G}=\mathbb{F}_{q}(x_1,N(x_2),\ldots,N(x_{m+1}),h_1,\ldots,h_{m-1})^{<L>}.$$
It also follows from Lemma \ref{le:invfieldO+evenG1} that
$$R:=\mathbb{F}_{q}[x_1,N(x_2),\ldots,N(x_{m+1}),h_1,\ldots,h_{m-1}]$$
is a polynomial ring, so $\mathbb{F}_{q}(x_1,N(x_2),\ldots,N(x_{m+1}),h_1,\ldots,h_{m-1})^{<L>}$ is the fraction field of $R^{<L>}$.
Now, $<L>$ is a group of order 2 and it is acting on $R$ such that it fixes the elements $x_1$,$N(x_2),\ldots,N(x_{m-1})$,$h_1$,\\
$h_2,\ldots,h_{m-1}$ and swaps $N(x_{m})$ with $N(x_{m+1})$.
It is known that the invariant ring for the symmetric group $\Sigma_2$ acting on $\mathbb{F}_{q}[X,Y]$ by interchanging $X$ with $Y$ is generated by $X+Y$ and $XY$ (see \cite{Smith95} Theorem 1.1.1).
Hence
$\mathbb{F}_{q}(x_1,N(x_2),\ldots,N(x_{m+1}),h_1,\ldots,h_{m-1})^{<L>}$ is generated by
\[x_1,N(x_2),\ldots,N(x_{m}),
N(x_{m})+N(x_{m+1}),N(x_{m})N(x_{m+1}),h_1,\ldots,h_{m-1}\]
and this finishes the proof.
\end{proof}

\subsection{The Invariant Field of a Sylow $p$-subgroup of $O^-(2m+2,q)$}
Here $\mathbb{F}=\mathbb{F}_q$ and $n=2m+2$. In \cite{Taylor92} we can see that the orthogonal group $O^-(2m+2,q)$ is the group of invertible matrices that preserve the quadratic form
$$Q(v)=\sum_{i=1}^{m}\alpha_{2m+2-i+1}\alpha_i+\alpha_{m+1}^2+\alpha_{m+1}\alpha_{m+2}+a\alpha_{m+2}^2,$$
where we chose $a$ such that the polynomial $X^2+X+a$ is irreducible in $\mathbb{F}_{q}[X]$.
Keeping in mind that now $n=2(m+1)$, for $k\geq 1$ define
$$h_k:=\Gamma_{k-1,a}.$$
Thus $h_1=\sum_{i=1}^mx_{2m+2-i+1}x_i+x_{m+1}^2+x_{m+1}x_{m+2}+ax_{m+2}^2$.
We prove that all $h_k$ are invariant under the action of $O^-(2m+2,q)$.

\begin{lemma}\label{le:hkinvO-}
For every $k\geq 1$, $h_k$ belongs to $\mathbb{F}_{q}[V]^{O^-(2m+2,q)}$.
\end{lemma}
\begin{proof}
We know from Lemma \ref{cor:StOpOmegaGammaLamba} that for $k>1$, $h_k$ is the $q^{k-2}$-th Steenrod operation of $h_{k-1}$ and so we only need to check that $h_1$ is invariant. Just as in the proof of Lemma \ref{le:hkinvO+}, this follows from the definition of $O^-(2m+2,q)$.
\end{proof}

Just as in the previous subsection, we study separately the cases when $q$ is odd and when it is even.
\begin{theorem}\label{the:invfieldO-odd}
Let $q$ be odd and  let $G$ be the Sylow $p$-subgroup of $O^-(2m+2,q)$ defined in Lemma \ref{le:sylowOrthogonaloddq}. The invariant field $\mathbb{F}_{q}(V)^G$ is generated by the polynomials $N(x_i)$, with $i=1,\ldots,m+2$, and the polynomials $h_j$, with $j=1,\ldots,m$, i.e.,
\[
\mathbb{F}_{q}(V)^G=\mathbb{F}_{q}(x_1,N(x_2),\ldots,N(x_{m+2}),h_1,\ldots,h_{m}).
\]
\end{theorem}
\begin{proof}
If in the proof of Theorem \ref{the:invfieldO+odd} we replace $m$ by $m+1$ and $\Omega_{0,1}$ by $\Gamma_{0,1}$, then we obtain a proof for this theorem.
\end{proof}

Now we assume that the characteristic of $\mathbb{F}_{q}$ is $2$ and consider the subgroup $G_1$ of $O^-(2m+2,q)$ given in Lemma \ref{le:sylowOrthogonalevenq}.

\begin{lemma}\label{le:invfieldO-evenG1}
The invariant field for $G_1$ is generated by the polynomials $N(x_i)$, with $i=1,\ldots,m+2$, and the polynomials $h_k$, with $k=1,\ldots,m$, i.e.,
\[
\mathbb{F}_{q}(V)^{G_1}=\mathbb{F}_{q}(x_1,N(x_2),\ldots,N(x_{m+2}),h_1,\ldots,h_{m}).
\]
\end{lemma}

\begin{proof}
If we replace $m$ by $m+1$ and use Theorem \ref{the:invfieldO-odd} instead of Theorem \ref{the:invfieldO+odd} in the proof of Lemma \ref{le:invfieldO+evenG1}, the we get a proof for the result here stated.
\end{proof}

\begin{theorem}\label{the:invfieldO-even}
Let $q$ be even and let $G$ be the Sylow $p$-subgroup of $O^-(2m+2,q)$, given by Lemma \ref{le:sylowOrthogonalevenq}. Then
$\mathbb{F}_{q}(V)^{G}$ is generated by
\[
x_1,N(x_2),\ldots,N(x_{m}),N(x_{m+1})^2+N(x_{m})N(x_{m+1}),N(x_{m+2}),h_1,\ldots,h_{m}.
\]
\end{theorem}
\begin{proof}
We use  similar arguments to those in the proof of Theorem \ref{the:invfieldO+even}. Here we also have
$\mathbb{F}_{q}(V)^{G}=(\mathbb{F}_{q}(V)^{G_1})^{<L_1>}.$
Applying Lemma \ref{le:invfieldO-evenG1} we get
$\mathbb{F}_{q}(V)^{G}=\mathbb{F}_{q}(x_1,N(x_2),\ldots,N(x_{m+2}),h_1,\ldots,h_{m})^{<L_1>},$
which is the fraction field of $R^{<L_1>}$ with
$R:=\mathbb{F}_{q}[x_1,N(x_2),\ldots,N(x_{m+2}),h_1,\ldots,h_{m}].$
Now, we can easily check that $<L_1>$ is a group of order 2 acting on $R$ by fixing  $x_1,N(x_2),\ldots,N(x_{m}),N(x_{m+2}),h_2,\ldots,h_{m-1}$ and mapping $N(x_{m+1})\mapsto N(x_{m+1})+N(x_{m})$.
Applying Theorem \ref{the:polynomialring} we can prove that the invariant ring of a group of order 2 acting on $\mathbb{F}_{q}[X,Y]$ such that it fixes $X$ and maps $Y$ to $Y+X$ is generated by $X$ and $Y^2+XY$.
Hence
$\mathbb{F}_{q}(x_1,N(x_2),\ldots,N(x_{m+2}),h_1,\ldots,h_{m})^{<L_1>}$ is generated by
\[x_1,N(x_2),\ldots,N(x_{m}),
N(x_{m+1})^2+N(x_{m})N(x_{m+1}),N(x_{m+2}),h_1,\ldots,h_{m})
\]
and the proof is complete.
\end{proof}

\subsection{The Invariant Field of a Sylow $p$-subgroup of $O(2m+1,q)$}
In \cite{Taylor92} we can see that the orthogonal group $O(2m+1,q)$ is the group of invertible matrices that preserve the quadratic form
$$Q(v)=\sum_{i=1}^{m}\alpha_{2m+1-i+1}\alpha_i+\alpha_{m+1}^2,$$
where we chose a basis for $V$ such that
$$v=\sum_{i=1}^m(\alpha_iu_i+\alpha_{2m+2-i+1}v_i)+\alpha_{m+1}w.$$

Consider the following family of polynomials: for $k\geq 1$ take
$$h_k:=\Gamma_{k-1,0}.$$
In particular, $h_1=\sum_{i=1}^mx_{n-i+1}x_i+x_{m+1}^2$.
The next lemma shows that all these polynomials are invariant under the action of $O(2m+1,q)$.

\begin{lemma}\label{hkinvariantO2m+1}
For all $k\geq 1$ the polynomials $h_k$ belong to $\mathbb{F}_{q}[V]^{O(2m+1,q)}$.
\end{lemma}
\begin{proof}
It follows from Lemma \ref{cor:StOpOmegaGammaLamba} that $h_k=\mathcal{P}^{q^{k-2}}(h_{k-1})$ for $k>1$. Hence it suffices to show that $h_1$ is an invariant polynomial. Again as in the proof of Lemma \ref{le:hkinvO+}, this follows from the definition of the group $O(2m+1,q)$.
\end{proof}

\begin{theorem}\label{the:invfieldO2m+1odd}
Let $G$ be the Sylow $p$-subgroup of $O(2m+1,q)$ given by Lemma \ref{le:sylowOrthogonaloddq} or by Lemma \ref{le:sylowOrthogonalevenq}. The invariant field $\mathbb{F}_{q}(V)^G$ is generated by the polynomials $N(x_i)$, with $i=1,\ldots,m+1$, and the polynomials $h_k$, with $k=1,\ldots,m$, i.e.,
\[
\mathbb{F}_{q}(V)^G=\mathbb{F}_{q}(x_1,N(x_2),\ldots,N(x_{m+1}),h_1,\ldots,h_{m}).
\]
\end{theorem}
\begin{proof}
First assume that $q$ is odd. The proof is analogous, for example, to the proofs of Theorems \ref{the:invfieldO+odd} or \ref{the:invariantfieldGUodd}.
In the same way we construct an abelian subgroup $H$ of $G$ which we then prove to be the subgroup $H^-$. Therefore Proposition  \ref{prop:minimaldegrees} tell us that $q^{m-k}$ is a lower bound for the minimal degree in $x_{m+1+k}$ of a polynomial in $R[m+1+k]^G$ for every $k\in \{1,\ldots,m\}$. Applying Propositions  \ref{prop:degreesphijOmegaGammaDelta}, Lemma \ref{cor:psilOmegaGammaDeltaalgebramemb}, Proposition \ref{prop:psil}-2 and Lemma \ref{hkinvariantO2m+1} we conclude that
$$\psi_{m-k}(\Gamma_{0,0})=\psi_{m-k}(h_1)\in R[m+1+k]^G$$
with degree $q^{m-k}$ in $x_{m+1+k}$.
For each $j\in \{1,\ldots,m+1\}$, and using the same argument as in the proof of \ref{the:invfieldO+odd}, we can show that the degree in $x_j$ of $N(x_j)$ is minimal among the elements of $R[j]^G$. Now, applying Theorem \ref{the:invfieldgen} completes the proof for odd $q$.
If $q$ is even, the proof is analogous to the previous one, but now we use Lemma \ref{le:invfieldO+evenG1} instead of Theorem \ref{the:invfieldO+odd}.
\end{proof}

\section{Proofs for Section 3}\label{sec:proofsylows}
In this section we present the proofs of Lemmas \ref{le:sylowGUandSp}, \ref{le:sylowOrthogonaloddq} and \ref{le:sylowOrthogonalevenq}.
We keep the notation of Section \ref{sec:sylowdescription}.
We start by determining the orders of the groups $\mathfrak G_{X_1,X_2}^{+}$ and $\mathfrak G_{X_1,X_2}^{-}$ with the following extra assumption, which will be satisfied in all
cases considered later:

{\bf Hypothesis (H):} If $\mathbb{F}=\mathbb{F}_q$ with $q$  even, then we assume that the diagonal entries of $B^TX_2B$ are equal to zero for every $B\in M(l\times n, \mathbb{F})$.

\begin{lemma}\label{lem:numberS}
Assume hypothesis (H) and for given $B$ let $\chi_B$ be the number of matrices $S$ satisfying:
$S+(\epsilon\bar{S}^T)=-B^TX_2\bar{B}$. Then $\chi_B$ is, independently of $B$, equal to
$$\left\{
\begin{array}{ll}
q^{n(n-1)}q^n& \textnormal{ if } \mathbb{F}=\mathbb{F}_{q^2}\\
q^{\frac{n(n-1)}{2}} & \textnormal{ if } \mathbb{F}=\mathbb{F}_{q}, q \textnormal{ odd } \textnormal{ and } \epsilon=\textnormal{``$+$"}\\
q^{\frac{n(n-1)}{2}}q^n & \textnormal{ if } \mathbb{F}=\mathbb{F}_{q}, q \textnormal{ even } \textnormal{ and }
\epsilon=\textnormal{``$+$"}\\
q^{\frac{n(n-1)}{2}}q^n & \textnormal{ if } \mathbb{F}=\mathbb{F}_{q}
\textnormal{ and } \epsilon=\textnormal{``-"}\\
\end{array}
\right..$$
\end{lemma}
\begin{proof}
We prove $(i)$ and $(ii)$ at the same time. It is not hard to see that the solution set of the matrix-equation
in the Lemma is non-empty: indeed,  if $\mathbb{F}=\mathbb{F}_q$ this follows from $1/2\in\mathbb{F}_q$
if $q$ is odd and from hypothesis $H$ if $q$ even. Let  $\mathbb{F}=\mathbb{F}_{q^2}$
and $Y=\epsilon\bar{Y}^T$ an arbitrary ``right hand side". Due to the surjectivity of the trace function,
there always exists $c\in\mathbb{F}$ with $c+\bar c=1$, so for $S:=cY$ we have $S+(\epsilon\bar{S}^T)=Y$.
Hence the number of choices for $S$ is the same as the number of solutions for $M+(\epsilon\bar{M}^T)=0$.
For this equation the number of solutions only depends on what happens to the diagonal entries of $M$ when we consider different fields. In fact, for the remaining ones the number of possibilities is always $r^{\frac{n(n-1)}{2}}$, where $r$ is the number of elements in $\mathbb{F}$. \\
When $\mathbb{F}=\mathbb{F}_{q}$, a simple argument give us the result. But if $\mathbb{F}=\mathbb{F}_{q^2}$, then we need to be more careful. Here the equation $M-\bar{M}^T=0$ implies that $m_{ii}=\bar{m}_{ii}$ for all $i$. Hence  $m_{ii}\in \mathbb{F}_{q}$ and there are $q^n$ choices for the elements in the diagonal of $M$.
Now, from $M+\bar{M}^T=0$ we obtain $m_{ii}+\bar{m}_{ii}=0$, i.e., each $m_{ii}$ belongs to the kernel of the
trace map, which has dimension 1 (see Lemma 10.1 in \cite{Taylor92}). So there will be $q^n$ different ways of choosing the elements in the diagonal of $M$.
\end{proof}

Note that for $\mathfrak G_{X_1,X_2}^{+}$ and $\mathfrak G_{X_1,X_2}^{-}$ the number of choices for $A$ and $B$ are the same. If $r$ is the number of elements in $\mathbb{F}$, then there are $r^{\frac{n(n-1)}{2}}$ choices for $A$ and $r^{ln}$ for $B$. Let $s$ be number of matrices $F\in U(l,\mathbb{F})$ satisfying $F^TX_2\bar{F}=X_2$. Applying Lemma \ref{lem:numberS} we obtain the orders of $\mathfrak G_{X_1,X_2}^{+}$ and $\mathfrak G_{X_1,X_2}^{-}$.
\begin{lemma}\label{lem:ordersG+-}
Let $s$ be as above and assume hypothesis (H). Then:
$$|\mathfrak G_{X_1,X_2}^{\epsilon}|=
 \left\{
\begin{array}{ll}
sq^{2n^2+(2l-1)n}& \textnormal{ if } \mathbb{F}=\mathbb{F}_{q^2}\\
sq^{n^2+(l-1)n} & \textnormal{ if } \mathbb{F}=\mathbb{F}_{q},  q \textnormal{ odd }  \textnormal{ and } \epsilon=
 \textnormal{``$+$"}\\
sq^{n^2+ln}  & \textnormal{ if } \mathbb{F}=\mathbb{F}_{q},  q \textnormal{ even}  \textnormal{ and } \epsilon=
 \textnormal{``$+$"} \\
\    & \textnormal{ or if } \mathbb{F}=\mathbb{F}_{q}  \textnormal{ and } \epsilon=
 \textnormal{``$-$"}\\
\end{array}
\right..$$
\end{lemma}

\hyperref[le:sylowGUandSp]{\bf{Proof of Lemma \ref*{le:sylowGUandSp}:}}
It is well known that up to equivalence there is only one non-degenerate hermitian form (see \cite{Taylor92}). Also in \cite{Taylor92} is shown that
\begin{equation}\label{orderGU(V)}
|GU(n,q^2)|=q^{\frac{n(n-1)}{2}}\prod_{i=1}^{n}(q^{i}-(-1)^i).
\end{equation}
Moreover, a basis can be choosen such that the matrix of the hermitian form is of type (\ref{eq:matrixX})
in Section \ref{sec:sylowdescription} with $\epsilon=$ ``$+$" and:
\begin{itemize}
    \item $X_1=J_{m-1}$ and $X_2=J_2$ for $GU(2m,q^2)$;
    \item $X_1=J_{m}$ and $X_2=[1]$  for $GU(2m+1,q^2)$.
\end{itemize}
Let $G_1=\mathfrak G_{J_{m-1},J_2}^{+}$ and $G_2=\mathfrak G_{J_{m},1}^{+}$ be the groups defined in Lemma \ref{le:sylowGUandSp}. We note that a matrix $F$ satisfies $F^TJ_2\bar{F}=J_2$ if and only if it is of the form
$\left(
\begin{tabular}[h]{cc}
$1$ & $0$ \\
$a$ &$1$
\end{tabular}
\right)
$ with $a+\bar{a}=0$. Hence there are $q$ different possibilities for $F$. Thus applying Lemma \ref{lem:ordersG+-}
we obtain that the order of $G_1$ is equal to $q^{2m^2-m}$. According to formula (\ref{orderGU(V)}) this is the order of a Sylow $p$-subgroup for  $GU(2m,q^2)$. Similarly we can show that $G_2$ is a Sylow  $p$-subgroup for  $GU(2m+1,q^2)$.
In \cite{Taylor92}, it is proven that up to equivalence there is only one non-degenerate alternating form which can be represented by the matrix (\ref{eq:matrixX}) of Section \ref{sec:sylowdescription} with $X_1=J_{m-1}$, $X_{2}=\left(
\begin{tabular}[h]{cc}
$0$ & $1$\\
$-1$ &$0$
\end{tabular}
\right)$
 and $\epsilon=$ ``$-$". Also,
\begin{equation}\label{Sp(V)order}
|Sp(2m,q)|=q^{m^2}\prod_{i=1}^{m}(q^{2i}-1).
\end{equation}
Let $G$ be the group given in Lemma \ref{le:sylowGUandSp}-3. Then any matrix $F\in U(2,q)$ satisfies $F^TJ_hF=J_h$ and the number of choices for $F$ is $q$.
Since hypothesis (H) is easily checked in this cases, applying Lemma \ref{lem:ordersG+-} shows that $G$ has order $q^{m^2}$. By formula (\ref{Sp(V)order}) this is the order of a Sylow $p$-subgroup and the proof is complete.

The \emph{orthogonal groups} are described as acting on $V=\mathbb{F}_q^n$ with $n\in\{2m,2m+1,2m+2\}$.
We set $v:=(\alpha_1,\cdots,\alpha_n)^T\in V$.

\begin{proposition}\label{prop:allO(V)}
 \begin{enumerate}
    \item[(i)] $O^+(2m,q)$ is the group of invertible matrices preserving the quadratic form
$Q(v)=\sum_{i=1}^{m}\alpha_{2m-i+1}\alpha_i.$
    \item[(ii)]  $O^-(2m+2,q)$ is the group of invertible matrices preserving the quadratic form
$$Q(v)=\sum_{i=1}^{m}\alpha_{2m+2-i+1}\alpha_i+\alpha_{m+1}^2+\alpha_{m+1}\alpha_{m+2}+a\alpha_{m+2}^2,$$
where $a$ is such that $X^2+X+a$ is irreducible in $\mathbb{F}_{q}[X]$.
    \item[(iii)]  $O(2m+1,q)$ is the group of invertible matrices preserving the quadratic form
$Q(v)=\sum_{i=1}^{m}\alpha_{2m+1-i+1}\alpha_i+\alpha_{m+1}^2 .$
\end{enumerate}
\end{proposition}

\begin{proof}
See \cite{Taylor92} Chapter 11.
\end{proof}

\begin{remark}\label{rewriteQ}
Define $X_n:=[\alpha_{1} \,\ldots\, \alpha_{n}]$ and let $J_{n}$ be the matrix given by (\ref{eq:Jm}) in
Section \ref{sec:sylowdescription}. Then we can rewrite the quadratic forms associated to each orthogonal group in the following way:
\begin{enumerate}
    \item[(i)]  $Q(v)=X_{m-1}J_{m-1}Y^T+\alpha_{m+1}\alpha_{m}$, where $Y:=[\alpha_{m+2} \,\ldots\, \alpha_{2m}]$, for $O^+(2m,q)$.
    \item[(ii)]  $Q(v)=X_mJ_mY^T+\alpha_{m+1}^2+\alpha_{m+1}\alpha_{m+2}+a\alpha_{m+2}^2$, where $Y:=[\alpha_{m+3} \,\ldots\, \alpha_{2m+2}]$, for $O^-(2m+2,q)$.
    \item[(iii)]  $Q(v)=X_mJ_mY^T+\alpha_{m+1}^2$, where $Y:=[\alpha_{m+2} \,\ldots\, \alpha_{2m+1}]$, for $O(2m+1,q)$.
\end{enumerate}
\end{remark}

The order of each orthogonal group is (see \cite{Taylor92}, pag. 140):
\begin{equation}\label{orderO+}
|O^+(2m,q)|=2q^{m(m-1)}(q^m-1)\prod_{i=1}^{m-1}(q^{2i}-1);
\end{equation}
\begin{equation}\label{orderO-}
|O^-(2m+2,q)|=2q^{m(m+1)}(q^{m+1}+1)\prod_{i=1}^{m}(q^{2i}-1);
\end{equation}
\begin{equation}
|O(2m+1,q)|=
\left\{
\begin{array}{ll}\label{orderO2m+1}
q^{m^2}\prod_{i=1}^{m}(q^{2i}-1) &  q  \textnormal{ even}\\
2q^{m^2}\prod_{i=1}^{m}(q^{2i}-1) & q \textnormal{ odd}
\end{array}
\right.
\end{equation}

\hyperref[le:sylowOrthogonaloddq]{\bf{Proof of Lemma \ref*{le:sylowOrthogonaloddq}:}}
This is similar to the proof of Lemma \ref{le:sylowGUandSp}. Since the characteristic of the field is odd, we obtain a Sylow $p$-subgroup by determining the subgroup of $U(n,q)$ preserving the bilinear form associated to the corresponding quadratic form.
Now, we can choose a basis such that (\ref{eq:matrixX}) in Section \ref{sec:sylowdescription} is the matrix of the bilinear form with:
\begin{itemize}
    \item $X_1=J_{m-1}$, $X_2=J_2$ and $X_3=J_{m-1}$ for $O^+(2m,q)$;
    \item $X_1=J_{m}$, $X_{2}=\left(
\begin{tabular}[h]{cc}
$2$ & $1$\\
$1$ &$2a$
\end{tabular}
\right)$ and $X_3=J_{m}$ for $O^-(2m+2,q)$;
    \item $X_1=J_{m}$, $X_2=[2]$ and $X_3=J_{m}$ for $O(2m+1,q)$.
\end{itemize}
Since $q$ is odd, hypothesis (H) holds and applying Lemma \ref{lem:ordersG+-}, we can show that, in each orthogonal group,  the group $\mathfrak G_{X_1,X_2}^{+}$ has the same order as a Sylow $p$-subgroup. This completes the proof.

\begin{lemma}\label{le:decompSqeven}
For $q$ even and $S\in M(n,q)$ there are unique matrices $S^\prime$ and $C$, with $S^\prime$ symmetric and $C$ upper triangular, such that $S=S^\prime+C$.
\end{lemma}
\begin{proof}
Define the symmetric matrix $S^\prime=[s_{ij}^\prime]$ by $s_{ij}^\prime:=s_{ji}$ for $i\leq j$,
$s_{ij}^\prime:=s_{ij}$ for $i\ge j$ and the upper triangular matrix $C:=[c_{ij}]$ by
$$
c_{ij}=
\left\{
\begin{array}{lr}
s_{ij}+s_{ji} & \textnormal{if } i<j\\
0 &  \textnormal{if } i\geq j
\end{array}
\right.
$$
Then we have $S=S^\prime+C$. The matrices $S^\prime$ and $C$ are unique because
$S=S_1^\prime+C_1=S_2^\prime+C_2$ implies that $S_1^\prime+S_2^\prime=C_1+C_2=0.$
\end{proof}

Let $J_2$ be the matrix $\left(
\begin{tabular}[h]{cc}
$0$ & $1$ \\
$1$ &$0$
\end{tabular}
\right).$

\begin{lemma}\label{lemaux:Oeven}
Let $S\in M(n,q)$ and $B\in M(2\times n,q)$. We also consider the row vectors $X=(\alpha_1,\ldots,\alpha_n)$, $Z=(z_1,z_2)$, $Y=(y_1,y_2)$ whose entries belong $\mathbb{F}_q$. Then
\begin{enumerate}
    \item[(i)] $\displaystyle XSX^T=\sum_{i=1}^ns_{ii}\alpha_i^2+\sum_{i=1}^n\sum_{j=1,i<j }^n (s_{ij}+s_{ji})\alpha_i\alpha_j$.
    \item[(ii)] If $q$ is even, $S+S^T=B^TJ_2B$ and $Z=Y+XB^T$ then
$$z_1z_2=y_1y_2+XB^TJ_2Y^T+XCX^T+\sum_{i=1}^nb_{1i}b_{2i}\alpha_i^2$$
where $C$ is the matrix defined in Lemma \ref{le:decompSqeven}.
\end{enumerate}
\end{lemma}

\begin{proof}
The result in $(i)$ is obvious.
Let us prove $(ii)$. It is not hard to check that $B^TJ_2B$ is a symmetric matrix and its entries are\linebreak
$b_{1i}b_{2j}+b_{1j}b_{2i}$, $i,j=1,\ldots,n$.
From $Z=Y+XB^T$ we get
\begin{eqnarray}
\nonumber z_1z_2 &=& \left(y_1+\sum_{i=1}^n b_{1i}\alpha_i\right)\left(y_2+\sum_{j=1}^n b_{2j}\alpha_j\right)\\
\nonumber &=&y_1y_2+XB^TJ_2Y^T+\sum_{i=1}^n\sum_{j=1}^n b_{1i}b_{2j}\alpha_i\alpha_j.
\end{eqnarray}
Since
\begin{eqnarray}
\nonumber \sum_{i=1}^n\sum_{j=1}^n b_{1i}b_{2j}\alpha_i\alpha_j&=&\sum_{i=1}^nb_{1i}b_{2i}\alpha_i^2+ \sum_{i=1}^n\sum_{j=1,j>i}^n (b_{1i}b_{2j}+b_{1j}b_{2i})\alpha_i\alpha_j\\
\nonumber &=&\sum_{i=1}^nb_{1i}b_{2i}\alpha_i^2+XCX^T.
\end{eqnarray}
this completes the proof of $(ii)$.
\end{proof}

\hyperref[le:sylowOrthogonalevenq]{\bf{Proof of Lemma \ref*{le:sylowOrthogonalevenq}:}}
In the even characteristic case we start by determining the subgroup $G$ of $U(n,q)$ preserving the bilinear form associated to the corresponding quadratic form. Then we compute the subgroup $G_1$ of $G$ whose elements preserve the quadratic form. For $O(2m+1,q)$, $G_1$ will be a Sylow $p$-subgroup. However, for the groups  $O^+(2m,q)$ and $O^-(2m+2,q)$ this is not the case and we shall need an additional element of order $2$ to  obtain a Sylow $p$-subgroup.

We start with the orthogonal group $O^+(2m,q)$. The matrix of the corresponding bilinear form can be represented as the matrix (\ref{eq:matrixX}) in Section \ref{sec:sylowdescription} with $X_1=J_{m-1}$, $X_2=J_2$ and $\epsilon=$ ``$+$".
Now we determine which elements $M$ in $\mathfrak G_{J_{m-1},J_{2}}^{+}$ satisfy $Q(Mv)=Q(v)$. By Remark \ref{rewriteQ} the quadratic form is $Q(v)=XJ_{m-1}Y^T+\alpha_{m+1}\alpha_{m}$ with $v=[X\;\;Z\;\;Y]^T$ and $Z:=[\alpha_{m}\;\;\alpha_{m+1}]$.
If $N$ is any element of $G$, then $Nv=[X^\prime\;\;Z^\prime\;\;Y^\prime]^T$ where
\[
\left\{
\begin{array}{l}
X^\prime=XA^T\\
Z^\prime=Z+XB^T\\
Y^\prime=XS^TA^{-1}J_{m-1}+ZD^T+YJ_{m-1}A^{-1}J_{m-1}
\end{array}
\right.
\]
Hence,
\begin{eqnarray}
\nonumber Q(Nv)&=&X^\prime J_{m-1}(Y^\prime)^T+\alpha_{m+1}^\prime\alpha_{m}^\prime\\
 &=&XJ_{m-1}Y^T+XSX^T+XB^TJ_2Z^T+\alpha_{m+1}^\prime\alpha_{m}^\prime.\label{proofO+even1}
\end{eqnarray}
Applying Lemma \ref{lemaux:Oeven} we get
$XSX^T=\sum_{i=1}^{m-1} s_{ii}\alpha_i^2+XCX^T,$ with $C$ from
Lemma \ref{le:decompSqeven}, and
$$\alpha_{m+1}^\prime\alpha_{m}^\prime=\alpha_{m+1}\alpha_{m}+XB^TJ_2Z^T+XCX^T+
\sum_{i=1}^{m-1}b_{1i}b_{2i}\alpha_i^2.$$
If we substitute these expressions in (\ref{proofO+even1}) we obtain
$Q(Nv)=\sum_{i=1}^{m-1}(s_{ii}+b_{1i}b_{2i})\alpha_i^2+Q(v)$ for all $v$, hence
$s_{ii}=b_{1i}b_{2i}$ for $i=1,\cdots,m-1$.
Hence $N$ belongs to $G_1$. It is not hard to check that all the matrices in $G_1$ preserve the quadratic form. Hence $G_1$ is the stabilizer subgroup of $Q$ in $\mathfrak G_{J_{m-1},J_{2}}^{+}$.
Since for elements of $G_1$, the entries $s_{ii}$ are defined by the matrix $B$, the proof of
Lemma \ref{lem:numberS} shows that $|G_1|= q^{m(m-1)}$.
\\
Now, we claim that $L$ (see (\ref{matrixLL1}) in Section \ref{sec:sylowdescription}) normalises the group $G_1$. So let $N\in G_1$. The product $LNL$ only changes the matrices $B$ and $D$ in $N$ to $B^\prime=J_2B$ and $D^\prime=DJ_2$, respectively. A straightforward calculation shows that $S+S^T=(B^\prime)^TJ_2B^\prime$ and $D^\prime=-J_{m-1}(A^{-1})^{T}(B^\prime)^TJ_2$. Hence $LNL\in G_1$ and this proves our claim. The order of the group generated by $G_1$ and $L$ is therefore $2q^{m(m-1)}$, which by formula (\ref{orderO+}) is the same as the order of a Sylow $p$-subgroup. This completes the proof of Lemma  \ref{le:sylowOrthogonalevenq}.2.
\\[1mm]
For $O^-(2m+2,q)$ we repeat similar steps as in the previous case. The quadratic form is
$Q(v)=XJ_{m}Y^T+ \alpha_{m+1}^2+\alpha_{m+1}\alpha_{m+2}+a\alpha_{m+2}^2$ with $v=[X\;\;Z\;\;Y]^T$ and $Z:=[\alpha_{m+1}\;\;\alpha_{m+2}]$.
Now the matrix of the bilinear form is the matrix (\ref{eq:matrixX},Section \ref{sec:sylowdescription}) with $X_1=J_{m}$, $X_2=J_2$ and $\epsilon=$ ``$+$". We get for an element  $N$ in $\mathfrak G_{J_{m},J_{2}}^{+}$ that
$$Q(Nv)=\sum_{i=1}^m(s_{ii}+b_{1i}^2+b_{1i}b_{2i}+ab_{2i}^2)\alpha_i^2+Q(v)$$
and therefore $N$ preserves the quadratic form if and only if $N$ is an element of  $G_1$. Thus $G_1$ is subgroup of order $q^{m(m+1)}$.

To prove that $L_1$ (see (\ref{matrixLL1})) normalises $G_1$ we just repeat the same argument as above and we use the fact that $(J_2^\prime)^TJ_2J_2^\prime=J_2$.
Hence the group generated by $G_1$ and $L_1$ has order $2q^{m(m+1)}$ which is actually the order of a Sylow $p$-subgroup for $O^-(2m+2,q)$ with $q$ even (see formula (\ref{orderO-})). This completes the proof of Lemma  \ref{le:sylowOrthogonalevenq}.3.

Finally, in $O(2m+1,q)$ the matrix of the bilinear form is the matrix (\ref{eq:matrixX}), Section \ref{sec:sylowdescription}) with $X_1=J_{m}$, $X_2=0$ and $\epsilon=$ ``$+$".
Remark \ref{rewriteQ} shows that $Q(v)=XJ_{m}Y^T+\alpha_{m+1}^2$ with $v=[X\;\;\alpha_{m+1}\;\;Y]^T$.
Hence, we obtain for an element $N$ in $\mathfrak G_{J_{m},0}^{+}$ that
$$Q(Nv)=\sum_{i=1}^m(s_{ii}+b_{1i}^2)\alpha_i^2+Q(v).$$
Hence $N$ preserves the quadratic form if and only if $N$ belongs to $G_1$. From this we can conclude that $G_1$ is a subgroup of order $q^{m^2}$, which implies that $G_1$ is a Sylow $p$-subgroup for $O(2m+1,q)$ with $q$ even (see formula (\ref{orderO2m+1})). This completes the proof  of Lemma  \ref{le:sylowOrthogonalevenq}.1.

\section{Proofs for section 4}

In this section we present the proofs of Lemmas \ref{cor:StOpOmegaGammaLamba}, \ref{cor:psilOmegaGammaDeltaalgebramemb}, \ref{le:phijingeneral} and Proposition \ref{prop:minimaldegrees}.

\begin{proposition}\label{prop:P(-1)omegagammadelta}
Let  $\Omega_{s,j}$, $\Gamma_{s,\lambda}$ and $\Delta_{s,\lambda}$ be the polynomials defined above. Then:
\begin{enumerate}
    \item $\mathcal{P}^\bullet(\Omega_{0,1})=\Omega_{0,1}^r-\Omega_{1,1}+\Omega_{0,1}$;
    \item $\mathcal{P}^\bullet(\Omega_{1,1})=\Omega_{1,1}^r-\Omega_{2,1}-2\Omega_{0,1}^r+\Omega_{1,1}$;
    \item $\mathcal{P}^\bullet(\Omega_{s,j})=\Omega_{s,j}^r-\Omega_{s+1,j}-\Omega_{s-1,j}^r+\Omega_{s,j}$ for  $s>0$ if $j=-1$ and $s>1$ if $j=1$;
    \item $\mathcal{P}^\bullet(\Gamma_{0,\lambda})=\Gamma_{0,\lambda}^r-\Gamma_{1,\lambda}+\Gamma_{0,\lambda}$;
    \item $\mathcal{P}^\bullet(\Gamma_{1,\lambda})=\Gamma_{1,\lambda}^r-\Gamma_{2,\lambda}-2\Gamma_{0,\lambda}^r+
\Gamma_{1,\lambda}$;
    \item $\mathcal{P}^\bullet(\Gamma_{s,\lambda})=\Gamma_{s,\lambda}^r-\Gamma_{s+1,\lambda}-
\Gamma_{s-1,\lambda}^r+\Gamma_{s,\lambda}$ for $s>1$;
    \item $\mathcal{P}^\bullet(\Lambda_{1,\lambda})=\Lambda_{1,\lambda}^{q^2}-\Lambda_{2,\lambda}-\Lambda_{1,\lambda}^q+\Lambda_{1,\lambda}$;
    \item $\mathcal{P}^\bullet(\Lambda_{s,\lambda})=\Lambda_{s,\lambda}^{q^2}-\Lambda_{s+1,\lambda}-
\Lambda_{s-1,\lambda}^{q^2}+\Lambda_{s,\lambda}$ for $s\geq 2$.
\end{enumerate}
\end{proposition}
\begin{proof}
Applying $\mathcal{P}^\bullet$ to $\Omega_{0,1}$ we obtain
\begin{eqnarray}
\nonumber \mathcal{P}^\bullet(\Omega_{0,1})&=&\sum_{i=1}^m\mathcal{P}^\bullet(x_{n-i+1})\mathcal{P}^\bullet(x_i)= \sum_{i=1}^m(x_{n-i+1}-x_{n-i+1}^r)(x_{i}-x_{i}^r)\\
\nonumber &=&\Omega_{0,1}-\Omega_{1,1}+\Omega_{0,1}^r
\end{eqnarray}
and 1 is proved. Now
\begin{eqnarray}
\nonumber &&\mathcal{P}^\bullet(\Omega_{s,j})=\sum_{i=1}^m(\mathcal{P}^\bullet(x_{n-i+1})^{r^s}\mathcal{P}^\bullet(x_i)
+j\mathcal{P}^\bullet(x_{n-i+1})\mathcal{P}^\bullet(x_i)^{r^s})\\
\nonumber &=&\sum_{i=1}^m((x_{n-i+1}^{r^s}-x_{n-i+1}^{r^{s+1}})(x_{i}-x_{i}^r)+j(x_{n-i+1}-x_{n-i+1}^r)(x_{i}^{r^s}-x_{i}^{r^{s+1}}))
\end{eqnarray}
and from this 2 and 3 follow.
Before proving $4$, 5 and $6$ note that by taking
$f_{s,\lambda}:=x_{m+1}^{r^{s}+1}+\lambda\,x_{m+2}^{r^{s}+1}$
we can write
$\Gamma_{0,\lambda}=\Omega_{0,1}+f_{0,\lambda},\quad \Gamma_{s,\lambda}=\Omega_{s,1}+2f_{s,\lambda}.$
Thus,
$$ \mathcal{P}^\bullet(\Gamma_{0,\lambda})=\mathcal{P}^\bullet(\Omega_{0,1})
+\mathcal{P}^\bullet(f_{0,\lambda}),\quad
\mathcal{P}^\bullet(\Gamma_{s,\lambda})=\mathcal{P}^\bullet(\Omega_{s,1})+
2\mathcal{P}^\bullet(f_{s,\lambda})$$
and therefore we just need to determine how $\mathcal{P}^\bullet$ acts on the polynomials $f_{s,\lambda}$. Following the same reasoning as in the beginning of the proof, we can show that
\begin{eqnarray}
\nonumber \mathcal{P}^\bullet(f_{0,\lambda})&=&f_{0,\lambda}^r-2f_{1,\lambda}+f_{0,\lambda};\\
\nonumber \mathcal{P}^\bullet(f_{s,\lambda})&=&f_{s,\lambda}^r-f_{s+1,\lambda}-
f_{s-1,\lambda}^r+f_{s,\lambda} \quad \textnormal{ for }   s>0.
\end{eqnarray}
Combining this with the results in 1, 2 and 3 we get 4, 5 and 6. Now we prove 7. Since in this case $\mathbb{F}=\mathbb{F}_{q^2}$, we have $r=q^2$ and so
\begin{eqnarray}
\nonumber \mathcal{P}^\bullet(\Lambda_{1,\lambda})&=&\sum_{i=1}^m(\mathcal{P}^\bullet(x_{n-i+1})^{q}\mathcal{P}^\bullet(x_i)
+\mathcal{P}^\bullet(x_{n-i+1})\mathcal{P}^\bullet(x_i)^{q})+\lambda \mathcal{P}^\bullet(x_{m+1}^{q+1})\\
\nonumber &=&\sum_{i=1}^m((x_{n-i+1}^{q}-x_{n-i+1}^{q^3})(x_{i}-x_{i}^{q^2})+(x_{n-i+1}-x_{n-i+1}^{q^2})(x_{i}^{q}-x_{i}^{q^3}))\\
\nonumber &+&\lambda (x_{m+1}^q-x_{m+1}^{q^3})(x_{m+1}-x_{m+1}^{q^2})\\
\nonumber &=&\Lambda_{1,\lambda}^{q^2}-\Lambda_{2,\lambda}-\Lambda_{1,\lambda}^q+\Lambda_{1,\lambda}.
\end{eqnarray}
A similar calculation proves 8.
\end{proof}

\hyperref[cor:StOpOmegaGammaLamba]{\bf{Proof of Lemma \ref*{cor:StOpOmegaGammaLamba}:}}
We will prove the result only for the polynomials $\Omega_{s,j}$.
For a homogeneous polynomial $f$ such that the $i$-th steenrod operation $\mathcal{P}^i(f)\neq 0$, we obtain ${\rm deg}(\mathcal{P}^i(f))={\rm deg}(f)+i(r-1)$. Thus, we just need to consider the degrees of the terms in $\mathcal{P}^\bullet(\Omega_{0,1})$ and $\mathcal{P}^\bullet(\Omega_{s,j})$. We have
$$\mathcal{P}^\bullet(\Omega_{0,1})=\mathcal{P}^0(\Omega_{0,1})-\mathcal{P}^1(\Omega_{0,1})+\mathcal{P}^2(\Omega_{0,1})
-\mathcal{P}^3(\Omega_{0,1})+\cdots=\Omega_{0,1}^r-\Omega_{1,1}+\Omega_{0,1}$$
by Proposition \ref{prop:P(-1)omegagammadelta}.
The degrees of $\Omega_{0,1}$, $\Omega_{1,1}$ and $\Omega_{0,1}^r$ are $2$, $r+1$ and $2r$, respectively. Comparing this with the degrees of  $\mathcal{P}^i(\Omega_{0,1})$, we get $\mathcal{P}^0(\Omega_{0,1})=\Omega_{0,1}$, $\mathcal{P}^1(\Omega_{0,1})=\Omega_{1,1}$ and $\mathcal{P}^2(\Omega_{0,1})=\Omega_{0,1}^r$.
Again by Proposition \ref{prop:P(-1)omegagammadelta} we get for $s>0$,
\begin{eqnarray}
\nonumber \mathcal{P}^\bullet(\Omega_{1,1})&=&\mathcal{P}^0(\Omega_{1,1})-\mathcal{P}^1(\Omega_{1,1})+\mathcal{P}^2(\Omega_{1,1})
-\cdots\\
\nonumber &=&\Omega_{1,1}^r-\Omega_{2,1}-2\Omega_{0,1}^r+\Omega_{1,1},\\
\nonumber \mathcal{P}^\bullet(\Omega_{s,j})&=& \Omega_{s,j}^r-\Omega_{s+1,j}-\Omega_{s-1,j}^r+\Omega_{s,j}.
\end{eqnarray}
Hence
\begin{itemize}
    \item $\mathcal{P}^0(\Omega_{s,j})=\Omega_{s,j}$;
    \item ${\rm deg}\,\Omega_{s-1,j}^r=r(r^{s-1}+1)$, ${\rm deg}\,\mathcal{P}^{1}(\Omega_{s,j})=r^s+1+r-1$ and therefore $\mathcal{P}^1(\Omega_{1,1})=2\Omega_{0,1}^r$ and  $\mathcal{P}^1(\Omega_{s,j})=\Omega_{s-1,j}^r$;
    \item ${\rm deg}\,\Omega_{s+1,j}=r^{s+1}+1$, ${\rm deg}\mathcal{P}^{r^s}(\Omega_{s,j})=r^s+1+r^s(r-1)$ and therefore $\mathcal{P}^{r^s}(\Omega_{s,j})=\Omega_{s+1,j}$;
    \item ${\rm deg}\,\Omega_{s,j}^r=r(r^{s}+1)$, ${\rm deg}\mathcal{P}^{r^s+1}(\Omega_{s,j})=r^s+1+(r^s+1)(r-1)$ and therefore $\mathcal{P}^{r^s+1}(\Omega_{s,j})=\Omega_{s,j}^r$;
\end{itemize}
Similar arguments prove the remaining results in the lemma.

The next proposition shows how the $\mathbb{F}$-algebra homomorphism $\psi_l$ acts on the polynomials $\Omega_{s,j}$, $\Gamma_{s,\lambda}$ and $\Lambda_{s,\lambda}$.
\begin{proposition}\label{psilOmegaGammaDelta}
For every $l\geq 1$, the following is true:
\begin{enumerate}
    \item $\psi_l(\Omega_{0,1})=\psi_{l-1}(\Omega_{0,1})^r-\psi_{l-1}(x_l)^{r-1}\psi_{l-1}(\Omega_{1,1})
+\psi_{l-1}(x_l)^{2(r-1)}\psi_{l-1}(\Omega_{0,1})$;
    \item $\psi_l(\Omega_{1,1})=\psi_{l-1}(\Omega_{1,1})^r-\psi_{l-1}(x_l)^{r-1}\psi_{l-1}(\Omega_{2,1})-
2\psi_{l-1}(x_l)^{r(r-1)}\psi_{l-1}(\Omega_{0,1})^r+\newline
\psi_{l-1}(x_l)^{(r+1)(r-1)}\psi_{l-1}(\Omega_{1,1})$;
    \item $\psi_l(\Omega_{s,j})=\psi_{l-1}(\Omega_{s,j})^r-\psi_{l-1}(x_l)^{r-1}\psi_{l-1}(\Omega_{s+1,j})-
\psi_{l-1}(x_l)^{r^s(r-1)}\psi_{l-1}(\Omega_{s-1,j})^r+\psi_{l-1}(x_l)^{(r^s+1)(r-1)}\psi_{l-1}(\Omega_{s,j})$ for  $s>0$ if $j=-1$ and $s>1$ if $j=1$;
    \item $\psi_l(\Gamma_{0,\lambda})=\psi_{l-1}(\Gamma_{0,\lambda})^r-\psi_{l-1}(x_l)^{r-1}
\psi_{l-1}(\Gamma_{1,\lambda})+\psi_{l-1}(x_l)^{2(r-1)}\psi_{l-1}(\Gamma_{0,\lambda})$;
    \item $\psi_l(\Gamma_{1,\lambda})=\psi_{l-1}(\Gamma_{1,\lambda})^r-\psi_{l-1}(x_l)^{r-1}
\psi_{l-1}(\Gamma_{2,\lambda})-2\psi_{l-1}(x_l)^{r(r-1)}\psi_{l-1}(\Gamma_{0,\lambda})^r+\newline
\psi_{l-1}(x_l)^{(r+1)(r-1)}\psi_{l-1}(\Gamma_{1,\lambda})$;
    \item $\psi_l(\Gamma_{s,\lambda})=\psi_{l-1}(\Gamma_{s,\lambda})^r-\psi_{l-1}(x_l)^{r-1}
\psi_{l-1}(\Gamma_{s+1,\lambda})-\psi_{l-1}(x_l)^{r^s(r-1)}\psi_{l-1}(\Gamma_{s-1,\lambda})^r+
\psi_{l-1}(x_l)^{(r^s+1)(r-1)}\psi_{l-1}(\Gamma_{s,\lambda})$ for $s\geq 2$;
    \item $\psi_l(\Lambda_{1,\lambda})=\psi_{l-1}(\Lambda_{1,\lambda})^{q^2}-\psi_{l-1}(x_l)^{q^2-1}
\psi_{l-1}(\Lambda_{2,\lambda})-\psi_{l-1}(x_l)^{q^3-q}\psi_{l-1}(\Lambda_{1,\lambda})^q\newline+
\psi_{l-1}(x_l)^{q^3+q^2-q-1}\psi_{l-1}(\Lambda_{1,\lambda})$;
    \item $\psi_l(\Lambda_{s,\lambda})=\psi_{l-1}(\Lambda_{s,\lambda})^{q^2}-\psi_{l-1}(x_l)^{q^2-1}
\psi_{l-1}(\Lambda_{s+1,\lambda})-\\
\psi_{l-1}(x_l)^{q^{2s-1}(q^2-1)}\psi_{l-1}(\Lambda_{s-1,\lambda})^{q^2}+
\psi_{l-1}(x_l)^{(q^{2s-1}+1)(q^2-1)}\psi_{l-1}(\Lambda_{s,\lambda})$  for $s\geq 2$.
\end{enumerate}
\end{proposition}
\begin{proof}
We only prove 1, 2 and 3. All the other statements can be proved by similar calculations. But we should remember that when we consider the polynomials $\Lambda_{s,\lambda}$, $r$ is equal to $q^2$.
According to Proposition \ref{prop:psil}-3, the $\mathbb{F}$-algebra homomorphism $\psi_l$ satisfies $\psi_l(x_i)=\psi_{l-1}(x_i)^r-\psi_{l-1}(x_l)^{r-1}\psi_{l-1}(x_i)$ for all $i$. For simplicity, let $T=\psi_{l-1}(x_l)$. Then $\psi_l(x_i)=\psi_{l-1}(x_i)^r-T^{r-1}\psi_{l-1}(x_i)$ and
\begin{eqnarray}
\nonumber \psi_l(\Omega_{0,1})&=&\sum_{i=1}^m\psi_l(x_{n-i+1})\psi_l(x_i)\\
\nonumber &=&\sum_{i=1}^m(\psi_{l-1}(x_{n-i+1})^r-T^{r-1}\psi_{l-1}(x_{n-i+1}))(\psi_{l-1}(x_i)^r-T^{r-1}\psi_{l-1}(x_i))\\
\nonumber &=&\psi_{l-1}(\Omega_{0,1})^r-T^{r-1}\psi_{l-1}(\Omega_{1,1})
+T^{2(r-1)}\psi_{l-1}(\Omega_{0,1})
\end{eqnarray}
which proves 1. Since
\begin{eqnarray}
\nonumber &&\psi_l(\Omega_{s,j})=\sum_{i=1}^m(\psi_l(x_{n-i+1})^{r^s}\psi_l(x_i)
+j\psi_l(x_{n-i+1})\psi_l(x_i)^{r^s})\\
\nonumber &=&\sum_{i=1}^m(\psi_{l-1}(x_{n-i+1})^{r^{s+1}}-T^{r^s(r-1)}\psi_{l-1}(x_{n-i+1})^{r^s})
(\psi_{l-1}(x_i)^r-T^{r-1}\psi_{l-1}(x_i))\\
\nonumber &+& j\sum_{i=1}^m(\psi_{l-1}(x_{n-i+1})^r-T^{r-1}\psi_{l-1}(x_{n-i+1}))
(\psi_{l-1}(x_{i})^{r^{s+1}}-T^{r^s(r-1)}\psi_{l-1}(x_{i})^{r^s}),
\end{eqnarray}
2 and 3 follow easily.
\end{proof}

\hyperref[cor:psilOmegaGammaDeltaalgebramemb]{\bf{Proof of Lemma \ref*{cor:psilOmegaGammaDeltaalgebramemb}:}}
We only prove the statement in $1$. We do this by induction on $l$. First we consider $\Omega_{s,j}$ with $s>0$ if $j=-1$ and $s>1$ if $j=1$.
For $l=1$, it follows from Proposition \ref{psilOmegaGammaDelta}-3 that
$$\psi_1(\Omega_{s,j})=\Omega_{s,j}^r-x_1^{r-1}\Omega_{s+1,j}-x_1^{r^s(r-1)}\Omega_{s-1,j}^r+x_1^{(r^s+1)(r-1)}\Omega_{s,j}$$
belongs to $\mathbb{F}[x_1,\Omega_{s-1,j},\Omega_{s,j},\Omega_{s+1,j}]$
Now, assume that the result is true for $l-1$.
Again from Proposition \ref{psilOmegaGammaDelta}-3 we get
\begin{eqnarray}
\nonumber \psi_l(\Omega_{s,j})&=&\psi_{l-1}(\Omega_{s,j})^r-\psi_{l-1}(x_l)^{r-1}\psi_{l-1}(\Omega_{s+1,j})- \psi_{l-1}(x_l)^{r^s(r-1)}\psi_{l-1}(\Omega_{s-1,j})^r\\
\nonumber &+& \psi_{l-1}(x_l)^{(r^s+1)(r-1)}\psi_{l-1}(\Omega_{s,j}).
\end{eqnarray}
By induction we have

$\psi_{l-1}(\Omega_{s,j})\in \mathbb{F}[x_1,\psi_1(x_2),\ldots,\psi_{l-2}(x_{l-1}),\Omega_{s-1,j},\Omega_{s,j},\Omega_{s+1,j},\ldots,\Omega_{s+l-1,j}]$;\\
$\psi_{l-1}(\Omega_{s+1,j})\in \mathbb{F}[x_1,\psi_1(x_2),\ldots,\psi_{l-2}(x_{l-1}),\Omega_{s,j},\Omega_{s+1,j},\ldots,\Omega_{s+l,j}]$;\\
$\psi_{l-1}(\Omega_{s-1,j})\in \mathbb{F}[x_1,\psi_1(x_2),\ldots,\psi_{l-2}(x_{l-1}),\Omega_{s-2,j},\Omega_{s-1,j},\Omega_{s,j},\ldots,\Omega_{s+l-2,j}]$ if  $s-1>0$. The arguments for $s=0$ and $s=1$ are similar.
Hence,
$$\psi_l(\Omega_{s,j})\in \mathbb{F}[x_1,\psi_1(x_2),\ldots,\psi_{l-1}(x_l),\Omega_{s-1,j},\Omega_{s,j},\Omega_{s+1,j},\ldots,\Omega_{s+l,j}].$$
Now, if in the previous argument, instead Proposition \ref{psilOmegaGammaDelta}-3, we use Proposition \ref{psilOmegaGammaDelta}-1 and Proposition \ref{psilOmegaGammaDelta}-2, then we get the result for $\Omega_{0,1}$ and $\Omega_{1,1}$, respectively. This proves 1.
Similarly, the other statements can be obtained also by induction on $l$.

Now we consider two families of subgroups $L_k^+$ and $L_k^-$ of $H^+$ and $H^-$, respectively.
For any matrix $A$ and $k\geq 1$ define:
\begin{itemize}
    \item $A^{(1)}=A$;
    \item $A^{(k)}$ is the matrix obtained from $A$ by fixing all the entries in the first $k-1$ rows equal to zero.
\end{itemize}
Let $k\in\{1,2,\ldots,t\}$. We represent by $L_k^+$ the subgroup of $H^+$ formed by the matrices
$$\left(
\begin{tabular}{c|c|c}
$I_t$ & $0$ & $0$ \\
\hline
$0$ & $I_d$ & $0$\\
\hline
$C^{+(k)}$ & $0$ & $I_t$
\end{tabular}
\right).$$
Replacing $C^{+(k)}$ by $C^{-(k)}$ in the elements of $L_k^+$, we obtain a subgroup of $H^-$ which we represent by $L_k^-$.
We want to determine the invariant rings $R[t+d+k]^{L_k^+}$ and $R[t+d+k]^{L_k^-}$ for all $k$. First, we need to see what happens to the entries of the matrices $C^{+}$ and $C^{-}$ when we take different fields.
\begin{lemma}\label{le:C+C-}
Consider the matrices $C^{+}$ and $C^{-}$. Then
\begin{enumerate}
    \item For $\mathbb{F}=\mathbb{F}_{q^2}$, we obtain
\begin{itemize}
    \item $c_{i,t-i+1}^+\in \mathbb{F}_{q}$ for all $i$ and $c_{i,j}^+\in \mathbb{F}_{q^2}$ if $j\neq t-i+1$.
    \item $c_{i,t-i+1}^-+\overline{c}_{i,t-i+1}^-=0$ for all $i$ and $c_{i,j}^-\in \mathbb{F}_{q^2}$ if $j\neq t-i+1$.
\end{itemize}
    \item For $\mathbb{F}=\mathbb{F}_{q}$, we get
\begin{itemize}
    \item $c_{i,j}^+\in \mathbb{F}_{q}$ for all $i$ and $j$.
    \item If $q$ is odd then $c_{i,t-i+1}^-=0$ for all $i$ and if $q$ is even then $c_{i,j}^-\in \mathbb{F}_{q}$ for all $i$ and $j$.
\end{itemize}
\end{enumerate}
\end{lemma}
\begin{proof}
All the statements follow from the fact that
$(i,j)=(t-j+1,t-i+1)\iff j=t-i+1$.
\end{proof}

\begin{lemma}\label{le:Fq2Nxm}
Let $G_1$ and $G_2$ be subgroups of $U(n,q^2)$ acting on $\mathbb{F}_{q^2}[V]$. Assume that for a fixed  $2\leq i\leq n$ and $l\leq i-1$, the orbit of $x_i$ under the action of:
\begin{itemize}
    \item $G_1$ is $\{ x_i+\sum_{j=1}^la_jx_j:a_1,\ldots a_{l-1}\in \mathbb{F}_{q^2}
\wedge a_l\in \mathbb{F}_{q}\}$;
   \item $G_2$ is $\{ x_i+\sum_{j=1}^la_jx_j:a_1,\ldots a_{l-1}\in \mathbb{F}_{q^2}
\wedge a_l+\bar{a}_l=0\}$.
\end{itemize}
Then the $G_\ell$-orbit product of $x_i$ for $\ell=1,2$ is given by:
\begin{enumerate}
    \item $N_{G_1}(x_i)=F_{l-1,q^2}(x_i)^q-F_{l-1,q^2}(x_l)^{q-1}F_{l-1,q^2}(x_i)$,
    \item $N_{G_2}(x_i)=F_{l-1,q^2}(x_i)^q+F_{l-1,q^2}(x_l)^{q-1}F_{l-1,q^2}(x_i)$.
\end{enumerate}
Moreover, both are homogeneous polynomials of degree $q^{2l-1}$.
\end{lemma}
\begin{proof}
We have
$$F_{l-1,q^2}(X)=\prod_{a_1,\ldots a_{l-1}\in \mathbb{F}_{q^2}}(X+a_1x_1+\ldots +a_{l-1}x_{l-1})$$
and it is homogeneous of degree $q^{2l-2}$.
Since $F_{l-1,q^2}(X)$ is $\mathbb{F}_{q^2}$-linear, replacing $X$ by $x_i+a_lx_l$ gives
$$F_{l-1,q^2}(x_i)+a_lF_{l-1,q^2}(x_l)=\prod_{a_1,\ldots, a_{l-1}\in \mathbb{F}_{q^2}}(x_i+a_lx_l+a_1x_1+\cdots +a_{l-1}x_{l-1}).$$
Therefore, we get
$$N(x_i)=\prod_{a_l\in \mathbb{F}_{q}}(F_{l-1,q^2}(x_i)+a_lF_{l-1,q^2}(x_l))=F_{l-1,q^2}(x_i)^q-
F_{l-1,q^2}(x_l)^{q-1}F_{l-1,q^2}(x_i)$$
for the orbit product of $x_i$ under the action of $G_1$.
Now, for the action of $G_2$ the orbit product of $x_i$ is given by
$$N(x_i)=\prod_{a_l+\bar{a}_l=0}(F_{l-1,q^2}(x_i)+a_lF_{l-1,q^2}(x_l)).$$
Note that $a_l+\bar{a}_l=0$ is equivalent to say that $a_l$ in the kernel of the standard
trace map $\mathbb{F}_{q^2}\to\mathbb{F}_q$, which is well known to be
a one dimensional vector space over $\mathbb{F}_{q}$ (see e.g. Lemma 10.1 in \cite{Taylor92}).
So if $c \notin \mathbb{F}_{q}$, then $c-\bar{c}$ is a basis for ${\rm ker }\,{\rm Tr}$ and
\begin{eqnarray}
\nonumber N(x_i)&=&\prod_{a\in \mathbb{F}_{q}}(F_{l-1,q^2}(x_i)+a(c-\bar{c})F_{l-1,q^2}(x_l))\\
\nonumber &=& F_{l-1,q^2}(x_i)^q-((c-\bar{c})F_{l-1,q^2}(x_l))^{q-1}F_{l-1,q^2}(x_i).
\end{eqnarray}
Since $(c-\bar{c})^{q-1}=-1$, the statement in $2$ is proved.
\end{proof}

\begin{theorem}\label{the:polynomialring}
Let $f_1,\cdots,f_n\in \mathbb{F}[V]^G$ be homogeneous invariants with $n=\dim V$. Then the following statements are equivalent:
\begin{enumerate}
    \item[(i)] $\mathbb{F}[V]^G=\mathbb{F}[f_1,\cdots,f_n]$.
    \item[(ii)] The $f_i$ are algebraically independent over $\mathbb{F}$ and $\displaystyle \prod_{i=1}^n\deg(f_i)$ is equal to $|G|$.
\end{enumerate}
\end{theorem}

\begin{proof}
See Proposition 16 in \cite{Kemper96} or Theorem 3.7.5 in \cite{DerKemp02}.
\end{proof}

\begin{proposition}\label{pro:Lk+}
Let $k\in \{1,\ldots,t\}$. Then
$$R[t+d+k]^{L_k^+}=\mathbb{F}[x_1,\ldots,x_{t+d},x_{t+d+1},\ldots,x_{t+d+k-1},N(x_{t+d+k})]$$
where $N(x_{t+d+k})$ is the orbit product of $x_{t+d+k}$ and in this variable its degree is:
\begin{itemize}
    \item $q^{2(t-k)+1}$ if $\mathbb{F}=\mathbb{F}_{q^2}$,
    \item $q^{t-k+1}$ if $\mathbb{F}=\mathbb{F}_{q}$.
\end{itemize}
\end{proposition}
\begin{proof}
For each $k$, the group $L_k^+$ acts on $R[t+d+k]$ in the following way: it fixes $x_i$ for all $i\leq t+d+k-1$ and
$$x_{t+d+k}\mapsto x_{t+d+k}+\sum_{j=1}^{t-k+1}c_{k,j}^+x_j.$$
Note that this defines an action of a subgroup $L$ of $U(t+d+k,\mathbb{F})$. Thus $R[t+d+k]^{L_k^+}=R[t+d+k]^{L}$.
We will show that the product of the degrees of
$$x_1,\ldots,x_{t+d},x_{t+d+1},\ldots,x_{t+d+k-1}, N(x_{t+d+k})$$
is equal to the order of $L$, which is the same as showing that the degree of  $N(x_{t+d+k})$ equals the order of $L$. Therefore, applying Theorem \ref{the:polynomialring} we obtain
$R[t+d+k]^{L}=\mathbb{F}[x_1,\ldots,x_{t+d},x_{t+d+1},\ldots,x_{t+d+k-1},N(x_{t+d+k})].$
First, we consider $\mathbb{F}=\mathbb{F}_{q^2}$. Applying Lemma \ref{le:C+C-}-1 we can conclude that the order of $L$ is $q^{2(t-k)+1}$. By Lemma \ref{le:Fq2Nxm}-1,
$$N(x_{t+d+k})=F_{t-k,q^2}(x_{t+d+k})^q-F_{t-k,q^2}(x_{t-k+1})^{q-1}F_{t-k,q^2}(x_{t+d+k})$$
and has degree $q^{2(t-k)+1}$.
Now, when $\mathbb{F}=\mathbb{F}_{q}$, the group $L$ has order  $q^{t-k+1}$ by Lemma \ref{le:C+C-}-2. In this case
\begin{eqnarray}
\nonumber N(x_{t+d+k})&=&\prod_{c_{k,1}^+\ldots c_{k,t-k+1}^+\in \mathbb{F}_{q}}x_{t+d+k}+\sum_{j=1}^{t-k+1}c_{k,j}^+x_j \\
\nonumber  &=&F_{t-k,q}(x_{t+d+k})^q-F_{t-k,q}(x_{t-k+1})^{q-1}F_{t-k,q}(x_{t+d+k})
\end{eqnarray}
and its order is $q^{t-k+1}$.
\end{proof}

We have a similar proposition for the groups $L_k^-$.
\begin{proposition}\label{pro:Lk-}
Let $k\in \{1,\ldots,t\}$. Then
$$R[t+d+k]^{L_k^-}=\mathbb{F}[x_1,\ldots,x_{t+d},x_{t+d+1},\ldots,x_{t+d+k-1},N(x_{t+d+k})]$$
where $N(x_{t+d+k})$ is the orbit product of $x_{t+d+k}$ and in this variable it has degree:
\begin{itemize}
    \item $q^{2(t-k)+1}$ if $\mathbb{F}=\mathbb{F}_{q^2}$,
    \item $q^{t-k}$ if $\mathbb{F}=\mathbb{F}_{q}$ and $q$ is odd,
    \item $q^{t-k+1}$ if $\mathbb{F}=\mathbb{F}_{q}$ and $q$ is even.
\end{itemize}
\end{proposition}
\begin{proof}
Just as in the proof of Proposition \ref{pro:Lk+}, the action of ${L_k^-}$ also defines an action of a subgroup $L$ of $U(t+d+k,\mathbb{F})$ on $R[t+d+k]$ and we just need to show that the degree of  $N(x_{t+d+k})$ is equal to the order of $L$.
When  $\mathbb{F}=\mathbb{F}_{q^2}$ it follows from Lemma \ref{le:C+C-}-1 that $L$ has order $q^{2(t-k)+1}$ and from Lemma \ref{le:Fq2Nxm}-2 that
$$N(x_{t+d+k})=F_{t-k,q^2}(x_{t+d+k})^q+F_{t-k,q^2}(x_{t-k+1})^{q-1}F_{t-k,q^2}(x_{t+d+k})$$
has degree $q^{2(t-k)+1}$. For $\mathbb{F}=\mathbb{F}_{q}$, we just apply Lemma \ref{le:C+C-}-2 to obtain that the order of $L$ is $q^{t-k}$ if $q$ is odd and $q^{t-k+1}$ if $q$ is even. In each case, the calculation of $N(x_{t+d+k})$ and its degree is straightforward.
\end{proof}

\hyperref[prop:minimaldegrees]{\bf{Proof of Proposition \ref*{prop:minimaldegrees}:}}
It follows from Propositions \ref{pro:Lk+} and \ref{pro:Lk-} that the degree of $N(x_{t+d+k})$ is the minimal degree in $x_{t+d+k}$ of a polynomial in $R[t+d+k]^{L_k^+}$ or $R[t+d+k]^{L_k^-}$. Since for each $k$ we have
$$R[t+d+k]^{H^+}\subset R[t+d+k]^{L_k^+}\quad \textnormal{and} \quad R[t+d+k]^{H^-}\subset R[t+d+k]^{L_k^-},$$
applying Propositions \ref{pro:Lk+} and \ref{pro:Lk-} completes the proof.

\hyperref[le:phijingeneral]{\bf{Proof of Lemma \ref*{le:phijingeneral}:}}
The groups $U_k$ act on $\mathbb{F}[x_1,x_2,\ldots,x_{n-1}]$ in the same way as  $U(n-1,\mathbb{F})$.
Hence
$$\mathbb{F}[x_1,x_2,\ldots,x_{n-1}]^{U_k}=\mathbb{F}[x_1,N(x_2),\ldots,N(x_{n-1})].$$
Now the order of $U_k$ is $|U(n-1,\mathbb{F})|s$ for some $s\in\mathbb{N}$.
We will show that the degree of $N(x_n)$ is equal to $s$ and then we apply Theorem \ref{the:polynomialring}.
Let $r$ be the number of elements in $\mathbb{F}$. First, we consider the group $U_1$. Therefore $s=r^{n-2}$. It is not hard to see that $N_{U_k}(x_n)=F_{n-2,r}(x_n)$ and consequently its degree is $r^{n-2}$.
For the group $U_2$, $s=q^{2(n-2)}q$ which is the degree of $N(x_n)$ according to Lemma \ref{le:Fq2Nxm}-2.

\end{document}